\documentclass[reqno,12pt]{amsart}

\usepackage[margin=1in]{geometry}
\usepackage{amsmath,amssymb,amsthm}
\usepackage[norelsize,ruled,vlined]{algorithm2e}
\usepackage{graphicx}
\usepackage{epstopdf}
\usepackage{url}

\theoremstyle{plain}

\newtheorem{theorem}{Theorem}
\newtheorem{lemma}[theorem]{Lemma}
\newtheorem{definition}[theorem]{Definition}

\begin{document}

\title{Phase retrieval from power spectra of masked signals}

\author{Afonso S.\ Bandeira}
\address{Program in Applied and Computational Mathematics, Princeton University, Princeton, NJ 08544}
\email{ajsb@math.princeton.edu}

\author{Yutong Chen}
\address{Program in Applied and Computational Mathematics, Princeton University, Princeton, NJ 08544}
\email{yutong@math.princeton.edu}

\author{Dustin G. Mixon}
\address{Department of Mathematics and Statistics, Air Force Institute of Technology, Wright-Patterson AFB, OH 45433}
\email{dustin.mixon@afit.edu}

\begin{abstract}
In diffraction imaging, one is tasked with reconstructing a signal from its power spectrum.
To resolve the ambiguity in this inverse problem, one might invoke prior knowledge about the signal, but phase retrieval algorithms in this vein have found limited success.
One alternative is to create redundancy in the measurement process by illuminating the signal multiple times, distorting the signal each time with a different mask. 
Despite several recent advances in phase retrieval, the community has yet to construct an ensemble of masks which uniquely determines all signals and admits an efficient reconstruction algorithm. 
In this paper, we leverage the recently proposed polarization method to construct such an ensemble.
We also present numerical simulations to illustrate the stability of the polarization method in this setting.
In comparison to a state-of-the-art phase retrieval algorithm known as PhaseLift, we find that polarization is much faster with comparable stability.
\end{abstract}

\maketitle

\section{Introduction}

In many applications, one wishes to reconstruct a signal from the magnitudes of its Fourier coefficients.
This problem is known as \textit{phase retrieval}, and it has been instrumental to many important scientific advances, including the Nobel Prize--winning work that leveraged X-ray diffraction to establish the double helix structure of DNA~\cite{WatsonC:53}.
Of course, given only the magnitudes of a signal's Fourier coefficients, one does not have enough information to recover the signal---while the Fourier transform is injective, the point-wise absolute value is not.
As such, one is inclined to use a priori knowledge of the signal, and hope it is then uniquely determined by the Fourier magnitudes. 
For example, to deduce the structure of DNA, Watson and Crick~\cite{WatsonC:53} applied certain chemical assumptions, along with a knowledge of van der Waals interactions between atoms.

In order to image more exotic molecules, such assumptions are difficult to apply, and so there has been quite a bit of work attempting to exploit more general assumptions (e.g., positivity or support constraints).
To account for this sort of prior information, the most popular phase retrieval algorithms are modifications of Gerchberg and Saxton's original approach~\cite{GerchbergS:72}, which alternates between the time and frequency domains, iteratively correcting the current guess by imposing time-domain assumptions or scaling Fourier coefficients to match the measured data.
Marchesini~\cite{Marchesini:07} surveys and compares the various modifications, but they all have a tendency to stall in local minima.

To address this issue, Cand\`{e}s, Eldar, Strohmer and Voroninski~\cite{CandesESV:11} proposed an alternative methodology whereby nonuniqueness is overcome not by prior information, but by additional illuminations.
For each illumination, a different mask (or grating) is used to distort the appearance of the object in question; in mathematical parlance, each mask acts as a multiplication operator on the desired signal before the Fourier transform.
Furthermore, Cand\`{e}s et al.\ constructed three masks such that the corresponding illuminations uniquely determine almost every signal up to a global phase factor; a fourth illumination is necessary to uniquely determine all signals~\cite{BandeiraCMN:13}.
However, the phase retrieval algorithms they propose are not guaranteed to reconstruct the desired signal from these three illuminations, even if the signal is uniquely determined by the illuminations. 

Let $F^*\colon\ell(\mathbb{Z}_M)\rightarrow\ell(\mathbb{Z}_M)$ denote the discrete Fourier transform (DFT) defined by
\begin{equation*}
(F^*x)(m)
:=\frac{1}{M}\sum_{m'\in\mathbb{Z}_M}x(m')e^{-2\pi imm'/M}
\qquad
\forall m\in\mathbb{Z}^M.
\end{equation*}
In this paper, we follow the lead of~\cite{CandesESV:11} to find masks $\{D_k\}_{k=0}^{K-1}$ (i.e., multiplication operators on $\ell(\mathbb{Z}_M)=\mathbb{C}^M$) such that every signal $x\in\ell(\mathbb{Z}_M)$ can be efficiently reconstructed, up to a global phase factor, from measurements of the form
\begin{equation*}
\{|F^*D_k^*x|^2\}_{k=0}^{K-1}=\{|(D_kF)^*x|^2\}_{k=0}^{K-1}=\{|\langle x,D_kf_m\rangle|^2\}_{k=0,}^{K-1}\ _{m=0}^{M-1},
\end{equation*}
where $f_m$ denotes the complex sinusoid $\{\frac{1}{M}e^{2\pi imm'/M}\}_{m'\in\mathbb{Z}_M}$.
Specifically, we will design the masks in such a way that allows for a new reconstruction technique to work, namely the polarization-based technique introduced in~\cite{AlexeevBFM:12}.
As established in~\cite{AlexeevBFM:12}, every signal $x$ can be stably reconstructed from $|\langle x,\varphi_i\rangle|^2$ if the measurement vectors $\varphi_i$ are constructed using a particular random process.
In this paper, we show how to fulfill the design criteria that measurement vectors be of the form $D_kf_m$ while still allowing for polarization-based reconstruction.
If anything, this demonstrates that polarization is flexible enough to provably accommodate certain measurement design requirements, which are bound to arise from various applications.

In the next section, we briefly review the main ideas behind the polarization-based technique, and we show how to apply them with Fourier masks.
In particular, we reduce the problem of building masks to a problem in additive combinatorics: 
Find a small subset of $\mathbb{Z}_M$ with small Fourier bias.
In Section~3, we use concentration-of-measure arguments to solve this problem, culminating in the construction of $\mathcal{O}(\log M)$ masks that uniquely determine every signal in $\mathbb{C}^M$ up to a global phase factor.
Next in Section~4, we compare the performance of polarization to a state-of-the-art phase retrieval algorithm known as PhaseLift~\cite{CandesESV:11,CandesSV:11}.
Here, we show that polarization provides a comparable amount of stability with Fourier masks, while being orders of magnitude faster.
We conclude in Section~5 with some remarks.

\section{How to polarize Fourier masks}

The main contribution of this paper is that the ideas in~\cite{AlexeevBFM:12} can be leveraged to uniquely measure signals with masked Fourier transforms.
In this section, we summarize these ideas and explain how we apply them.
In~\cite{AlexeevBFM:12}, the signal $x$ is measured with two ensembles of vectors $\Phi_V,\Phi_E\subseteq\mathbb{C}^M$.
Here, $\Phi_V$ has the property that every subcollection of size $M$ spans $\mathbb{C}^M$; such ensembles are called \textit{full spark frames}, of which there are several known constructions~\cite{AlexeevCM:12,PuschelK:05}.
To determine $\Phi_E$, first define a graph $G=(V,E)$ whose vertices $i\in V$ index the measurement vectors $\varphi_i\in \Phi_V$.
Then each edge $(i,j)\in E$ corresponds to three members of $\Phi_E$, namely $\{\varphi_i+\omega^r\varphi_j\}_{r=0}^2$, where $\omega=e^{2\pi i/3}$.
These edges are useful because they determine the relative phase between the inner products $\langle x,\varphi_i\rangle$ and $\langle x,\varphi_j\rangle$.
Specifically, a version of the polarization identity gives that
\begin{equation}
\label{eq.polarization trick}
\overline{\langle x,\varphi_i\rangle}\langle x,\varphi_j\rangle
=\frac{1}{3}\sum_{r=0}^2\omega^r|\langle x,\varphi_i\rangle+\omega^{-r}\langle x,\varphi_j\rangle|^2=\frac{1}{3}\sum_{r=0}^2\omega^r|\langle x,\varphi_i+\omega^r\varphi_j\rangle|^2,
\end{equation}
and so the relative phase is calculated by normalizing this quantity (provided it's nonzero).
On the other hand, if an inner product $\langle x,\varphi_i\rangle$ is zero, then for every vertex $j$ adjacent to $i$ in $G$, the members of $\Phi_E$ corresponding to $(i,j)$ are useless, since the relative phases they are intended to capture are not well defined.
Indeed, $x$ being orthogonal to $\varphi_i$ has the effect of deleting the vertex $i$ (and its edges) from the graph $G$.
However, since $\Phi_V$ is full spark, $x$ is orthogonal to at most $M-1$ vertex vectors, thereby limiting the damage to our graph.
Furthermore, for any remaining connected component $S\subseteq V$ of size $\geq M$, we can propagate the relative phases to determine $\{\langle x,\varphi_i\rangle\}_{i\in S}$ up to a global phase factor, and since $\{\varphi_i\}_{i\in S}$ necessarily span $\mathbb{C}^M$, we can determine $x$ from $\{\langle x,\varphi_i\rangle\}_{i\in S}$ by least-squares estimation.
Interestingly, there will always be a connected component of size $\geq M$ after the removal of any $M-1$ vertices, provided the graph $G$ is regular with a sufficiently large \textit{spectral gap}. 
To be clear, the spectral gap of a $d$-regular graph $G$ is determined by the eigenvalues $\lambda_1\geq\lambda_2\geq\cdots$ of its adjacency matrix:
\begin{equation*}
\operatorname{spg}(G)=\frac{1}{d}\Big(\lambda_1-\max_{i\neq 1}|\lambda_i|\Big),
\end{equation*}
and the existence of the requisite connected component is given by the following classical result in spectral graph theory (e.g., see Lemma~5.2 of~\cite{HarshaB:online}):

\begin{lemma}
\label{lemma.spectral gap vs connectivity}
Consider a $d$-regular graph $G$ of $n$ vertices.
For all $\varepsilon\leq\operatorname{spg}(G)/6$, removing any $\varepsilon dn$ edges from $G$ results in a connected component of size $\geq(1-2\varepsilon/\operatorname{spg}(G))n$.
\end{lemma}

In summary, in order to apply the polarization trick in the setting of masked DFTs, it suffices to (i) construct a full spark ensemble $\Phi_V$ with masked DFTs, and (ii) construct a graph $G$ with a sufficiently large spectral gap and whose corresponding edge vectors $\Phi_E$ can also be implemented with masked DFTs.
We quickly address (i) with the following result:

\begin{lemma}
\label{lemma.full spark phi_v}
For each $k\in\{0,\ldots,K-1\}$, pick some nonzero $\alpha_k\in\mathbb{C}$ and define $D_k:=\operatorname{diag}\{\alpha_k^m\}_{m=0}^{M-1}$.
Let $f_m$ denote the complex sinusoid $\{\frac{1}{M}e^{2\pi imm'/M}\}_{m'\in\mathbb{Z}_M}$
Then $\{D_kf_m\}_{k=0,}^{K-1}\ \!_{m=0}^{M-1}$ is full spark if and only if $\alpha_k\alpha_{k'}^{-1}$ is not an $m$th root of unity for any pair of distinct $k,k'\in\{0,\ldots,K-1\}$.  
\end{lemma}

\begin{proof}
The $M\times KM$ matrix whose columns are $D_kf_m$ has Vandermonde form, as does each of its $M\times M$ submatrices.
These submatrices are all invertible precisely when their determinants are nonzero, that is, when $\{\alpha_ke^{2\pi i m/M}\}_{k=0,}^{K-1}\ \!_{m=0}^{M-1}$ are distinct (by the Vandermonde determinant formula).
The lemma immediately follows.
\end{proof}

The condition above is satisfied with probability $1$ if the $\alpha_k$'s are drawn uniformly at random from the complex unit circle.
For a deterministic alternative, it suffices to take $\alpha_k=e^{2\pi i k/KM}$.
Now that we have established how to construct masks $D_k$ such that $\Phi_V=\{D_kf_m\}_{k=0,}^{K-1}\ \!_{m=0}^{M-1}$ is full spark, we turn to solving (ii).
Before describing our construction of $G$, we first motivate it by illustrating how polarized combinations of our vertex vectors can be expressed using masked DFTs:

\begin{lemma}
\label{lemma.modulation trick}
Take $\omega=e^{2\pi i/3}$, let $E=\operatorname{diag}\{e^{2\pi im/M}\}_{m=0}^{M-1}$ denote the modulation operator, and consider $\Phi_V=\{D_kf_m\}_{k=0,}^{K-1}\ \!_{m=0}^{M-1}$, as defined in Lemma~\ref{lemma.full spark phi_v}.
Then
\begin{equation*}
D_kf_m+\omega^rD_{k'}f_{m'}
=(D_k+\omega^rE^{m'-m}D_{k'})f_m
\end{equation*}
for all $k,k'\in\{0,\ldots,K-1\}$ and $m,m'\in\mathbb{Z}_M$.
\end{lemma}

\begin{proof}
Consider the $\ell$th entry of $D_{k'}f_{m'}$:
\begin{equation*}
(D_{k'}f_{m'})(\ell)
=\alpha_{k'}^\ell\cdot \frac{1}{M}e^{2\pi im'\ell/M}
=e^{2\pi i(m'-m)\ell/M}\cdot\alpha_{k'}^\ell\cdot \frac{1}{M}e^{2\pi im\ell/M}
=(E^{m'-m}D_{k'}f_m)(\ell).
\end{equation*}
Then $D_kf_m+\omega^rD_{k'}f_{m'}=D_kf_m+\omega^rE^{m'-m}D_{k'}f_m=(D_k+\omega^rE^{m'-m}D_{k'})f_m$.
\end{proof}

When implementing the modulation trick of Lemma~\ref{lemma.modulation trick} using DFTs, we will have to fix $m'-m$ (so as to apply a fixed mask) and let $m$ vary (since we will mask an entire DFT).
As such, we intend to use auxiliary masks of the form $\{D_k+\omega^rE^aD_{k'}\}_{r=0}^2$, thereby drawing an edge between every $(k,m)$ and $(k',m')$ such that $m'-m=a\bmod M$.  
This informs our decision of how to construct the graph $G$, since it enables a masked-DFT implementation:

\begin{definition}
\label{definition.graph construction}
We construct $G$ as follows:
Pick $A\subseteq\mathbb{Z}_M$ such that $0\not\in A$ and $A=-A$.
Then $(k,m)$ and $(k',m')$ are adjacent precisely when $m'-m\in A$.
\end{definition}

Note that for each $a\in A$, the corresponding edges between $\{(k,m)\}_{m\in\mathbb{Z}_M}$ and $\{(k',m')\}_{m'\in\mathbb{Z}_M}$ are implemented with the masks $\{D_k+\omega^rE^aD_{k'}\}_{r=0}^2$.
As such, $A$ is the set of modulations used to combine the $K$ original masks $D_k$ into $3\binom{K+1}{2}|A|$ auxiliary masks.
Also note that $G$ has no loops since $0\not\in A$, $G$ is not directed since $A=-A$, and furthermore every vertex has degree $K|A|$, so $G$ is regular.
It remains to show that $G$ has a sufficiently large spectral gap.
To simplify this analysis, we first relate the spectral gap to a fundamental concept in additive combinatorics called the \textit{Fourier bias}.
As in~\cite{TaoV:06}, take $A\subseteq\mathbb{Z}_M$ and let $\mathbf{1}_A$ denote the characteristic function of $A$.
Then the Fourier bias of $A$ is defined as follows:
\begin{equation*}
\|A\|_u:=\max_{m\neq 0}|(F^*\mathbf{1}_A)(m)|.
\end{equation*}
Fourier bias is used in additive combinatorics to measure pseudorandomness; here, the main idea is that correlation with any complex sinusoid indicates regularity, which is not typically exhibited by random sets.
As indicated earlier, the spectral gap of $G$ is intimately related to the Fourier bias of $A$:

\begin{lemma}
\label{lemma.spectral gap vs fourier bias}
Consider the graph $G$ defined in Definition~\ref{definition.graph construction}.
The spectral gap of $G$ is
\begin{equation*}
\operatorname{spg}(G)
=1-\frac{M}{|A|}\|A\|_u,
\end{equation*}
where $\|\cdot\|_u$ denotes Fourier bias.
\end{lemma}

\begin{proof}
To determine the spectral gap, we first determine the adjacency matrix $W$ of $G$.
For any pair $k,k'\in\{0,\ldots,K-1\}$, the adjacency rule for $(k,m)$ and $(k',m')$ is whether $m'-m\in A$.
As such, the $(k,k')$th block of $W$ is circulant, with each row being a translation of $\mathbf{1}_A$.
This gives the expression $W=J\otimes\operatorname{circ}(\mathbf{1}_A)$, where $J$ is the $K\times K$ all-ones matrix and $\otimes$ denotes the Kronecker product.
One useful property of the Kronecker product is that its eigenvalues are products of eigenvalues:
\begin{equation*}
\lambda_{(i,j)}(W)=\lambda_i(J)\lambda_j(\operatorname{circ}(\mathbf{1}_A)).
\end{equation*}
Here, we note that the only nontrivial eigenvalue of $J$ is $K$, and the eigenvalues of $\operatorname{circ}(\mathbf{1}_A)$ are the entries of the Fourier transform $MF^*\mathbf{1}_A$.
As such, the nontrivial eigenvalues of $W$ are given by
\begin{equation*}
\lambda_{(1,m)}(W)
=\lambda_1(J)\lambda_{m}(\operatorname{circ}(\mathbf{1}_A))
=KM(F^*\mathbf{1}_A)(m)
=K\sum_{m'\in\mathbb{Z}_M}\mathbf{1}_A(m')e^{-2\pi i mm'/M}.
\end{equation*}
By the triangle inequality, we then have
\begin{equation*}
|\lambda_{(1,m)}(W)|
=K\bigg|\sum_{m'\in\mathbb{Z}_M}\mathbf{1}_A(m')e^{-2\pi i mm'/M}\bigg|
\leq K\sum_{m'\in\mathbb{Z}_M}|\mathbf{1}_A(m')|
=K|A|,
\end{equation*}
with equality when $m=0$.
This means the largest eigenvalue of $W$ is $\lambda_1=K|A|$, which corresponds to the all-ones eigenvector, and so the spectral gap of $G$ is
\begin{equation*}
\operatorname{spg}(G)
=\frac{1}{d}\Big(\lambda_1-\max_{i\neq 1}|\lambda_i|\Big)
=\frac{1}{K|A|}\Big(K|A|-\max_{m\neq 0}|KM(F^*\mathbf{1}_A)(m)|\Big)
=1-\frac{M}{|A|}\|A\|_u,
\end{equation*}
as claimed.
\end{proof}

At this point, we have reduced the problem of finding injective Fourier masks to finding a small set $A\subseteq\mathbb{Z}_M$ of sufficiently small Fourier bias (i.e., a problem in additive combinatorics).
In the next section, we use the following random process to construct this set:
Draw $B\subseteq\mathbb{Z}_M$ at random and take $A:=B\cup(-B)\setminus\{0\}$.
With the appropriate choice of distribution, we construct $A$ of size $\Theta(\log M)$ and Fourier bias $\mathcal{O}(\frac{\log M}{M})$, thereby producing a graph $G$ with a spectral gap that is bounded away from zero (see Theorem~\ref{theorem.bound on spectral gap}).
We also show that the spectral gap is bounded away from zero only if $A$ has size $\Omega(\log M)$, thereby demonstrating the optimality of our analysis (see Theorem~\ref{theorem.log is necessary}).
Overall, letting $K$ be a sufficiently large constant, the techniques of this paper use a total of $K+3\binom{K+1}{2}|A|=\Theta(\log M)$ Fourier masks to uniquely determine any signal (see Theorem~\ref{theorem.main result}).

\section{Small sets with small Fourier bias}

In the previous section, we reduced the main problem of this paper to one in additive combinatorics:
Find a small set $A\subseteq\mathbb{Z}_M$ with small Fourier bias $\|A\|_u=\max_{m\neq 0}|(F^*\mathbf{1}_A)(m)|$.
In our case, we want $A$ to satisfy two additional properties: $0\not\in A$ and $A=-A$.
To account for these, we will approach this problem by first drawing $B\subseteq\mathbb{Z}_M$ at random, and then taking $A:=B\cup(-B)\setminus\{0\}$.
We start with the following lemma:

\begin{lemma}
\label{lemma.no collisions}
Suppose the entries of $\mathbf{1}_B$ are independent, identical Bernoulli random variables with mean $\frac{c\log M}{M}$.
Then the following simultaneously hold with high probability:
\begin{itemize}
\item[(i)] $B\cap(-B)=\emptyset$
\item[(ii)] $0\not\in B$
\item[(iii)] $\frac{1}{2}c\log M\leq|B|\leq \frac{3}{2}c\log M$
\end{itemize}
\end{lemma}

\begin{proof}
First, a union bound gives
\begin{align*}
\operatorname{Pr}\Big(B\cap(-B)\neq\emptyset\Big)
&=\operatorname{Pr}\Big(\exists m\in\mathbb{Z}_M \mbox{ s.t. }m,-m\in B\Big)\\
&\leq M\operatorname{Pr}\Big(m\in B\Big)\operatorname{Pr}\Big(-m\in B\Big)
=\frac{c^2\log^2M}{M}.
\end{align*}
Next, we have $\operatorname{Pr}(0\in B)=\frac{c\log M}{M}$.
For (iii), we apply a multiplicative form of the Chernoff bound (see (6) and (7) in~\cite{HagerupR:90}):
Let $X$ be a sum of independent 0-1 random variables.
Then
\begin{equation*}
\operatorname{Pr}\Big(\big|X-\mathbb{E}[X]\big|\geq\delta\mathbb{E}[X]\Big)
\leq 2e^{-\delta^2\mathbb{E}[X]/3}
\end{equation*}
for any choice $0<\delta<1$.
Here, we let $X:=\sum_{m\in\mathbb{Z}_M}\mathbf{1}_B(m)=|B|$ and take $\delta=1/2$.
Then
\begin{equation*}
\operatorname{Pr}\bigg(\big||B|-c\log M\big|\geq\frac{1}{2}c\log M\bigg)
\leq 2e^{-(c\log M)/12}
=2M^{-c/12}.
\end{equation*}
The result then follows from a union bound.
\end{proof}

Note that in the event of Lemma~\ref{lemma.no collisions}, we have $\mathbf{1}_A=\mathbf{1}_B+\mathbf{1}_{-B}$, and so
\begin{align}
\nonumber
\|A\|_u
=\max_{m\neq 0}|(F^*\mathbf{1}_A)(m)|
&\leq\max_{m\neq 0}\Big(|(F^*\mathbf{1}_B)(m)|+|(F^*\mathbf{1}_{-B})(m)|\Big)\\
\label{eq.bound on fourier bias}
&\leq\max_{m\neq 0}|(F^*\mathbf{1}_B)(m)|+\max_{m\neq 0}|(F^*\mathbf{1}_{-B})(m)|
=2\|B\|_u,
\end{align}
where the last step applied a complex conjugate to $(F^*\mathbf{1}_{-B})(m)$.
As such, it suffices to show that random sets $B$ have small Fourier bias:

\begin{lemma}
\label{lemma.small fourier bias}
Pick $c\geq4$ and suppose the entries of $\mathbf{1}_B$ are independent, identical Bernoulli random variables with mean $\frac{c\log M}{M}$.
Then $\|B\|_u\leq3\sqrt{c}\cdot\frac{\log M}{M}$ with high probability.
\end{lemma} 

To prove this lemma, we will apply the following version of the Chernoff bound:

\begin{lemma}
\label{lemma.chernoff bound}
Assume that $\{Z_i\}_{i=1}^n$ are jointly independent complex random variables where $|Z_i-\mathbb{E}[Z_i]|\leq 1$ for all $i$.
Take $Z:=\sum_{i=1}^nZ_i$ and denote $\sigma:=\sqrt{\operatorname{Var}(Z)}$.
Then
\begin{equation*}
\operatorname{Pr}\Big(\big|Z-\mathbb{E}[Z]\big|\geq t\sigma\Big)
\leq 4\max\Big\{e^{-t^2/8},e^{-t\sigma/2\sqrt{2}}\Big\}
\end{equation*}
for any $t>0$.
\end{lemma}

\begin{proof}
Let $X$ and $Y$ denote the real and imaginary parts of $Z$.
Then a union bound gives
\begin{align*}
\operatorname{Pr}\Big(\big|Z-\mathbb{E}[Z]\big|\geq t\sigma\Big)
&=\operatorname{Pr}\Big(\big|X-\mathbb{E}[X]\big|^2+\big|Y-\mathbb{E}[Y]\big|^2\geq t^2\sigma^2\Big)\\
&\leq\operatorname{Pr}\bigg(\big|X-\mathbb{E}[X]\big|\geq \frac{t\sigma}{\sqrt{2}}\bigg)+\operatorname{Pr}\bigg(\big|Y-\mathbb{E}[Y]\big|\geq \frac{t\sigma}{\sqrt{2}}\bigg).
\end{align*}
Denote $\sigma_X:=\sqrt{\operatorname{Var}(X)}$ and $\sigma_Y:=\sqrt{\operatorname{Var}(Y)}$.
Since $\operatorname{Var}(X)+\operatorname{Var}(Y)=\operatorname{Var}(Z)$, we then have $\sigma_X\leq\sigma$ and $\sigma_Y\leq\sigma$, and so
\begin{equation*}
\operatorname{Pr}\Big(\big|Z-\mathbb{E}[Z]\big|\geq t\sigma\Big)
\leq\operatorname{Pr}\bigg(\big|X-\mathbb{E}[X]\big|\geq \frac{t\sigma_X}{\sqrt{2}}\bigg)+\operatorname{Pr}\bigg(\big|Y-\mathbb{E}[Y]\big|\geq \frac{t\sigma_Y}{\sqrt{2}}\bigg).
\end{equation*}
The result then follows from applying Theorem~1.8 of~\cite{TaoV:06} to both terms of the right-hand side.
\end{proof}

\begin{proof}[Proof of Lemma~\ref{lemma.small fourier bias}]
Fix $m\in\mathbb{Z}_M$, $m\neq0$, and take $Z_{m'}:=\frac{1}{2}\mathbf{1}_B(m')e^{-2\pi imm'/M}$ for each $m'\in\mathbb{Z}_M$.
Then $\{Z_{m'}\}_{m'\in\mathbb{Z}_M}$ are jointly independent with 
\begin{equation*}
|Z_{m'}-\mathbb{E}[Z_{m'}]|
\leq |Z_{m'}|+|\mathbb{E}[Z_{m'}]|
\leq 1.
\end{equation*}
As such, $\{Z_{m'}\}_{m'\in\mathbb{Z}_M}$ satisfy the conditions of Lemma~\ref{lemma.chernoff bound}, and so we now seek the expected value and variance of 
\begin{equation*}
Z
:=\sum_{m'\in\mathbb{Z}_M}Z_{m'}
=\frac{1}{2}\sum_{m'\in\mathbb{Z}_M}\mathbf{1}_B(m')e^{-2\pi imm'/M}.
\end{equation*}
Take $p:=\frac{c\log M}{M}$ to be the probability that $\mathbf{1}_B(m')$ is $1$.
Then for each $S\subseteq\mathbb{Z}_M$, the probability that $B=S$ is $p^{|S|}(1-p)^{M-|S|}$.
This leads to the following calculation:
\begin{align*}
\mathbb{E}[Z]
&=\sum_{S\subseteq\mathbb{Z}_M}\bigg(\frac{1}{2}\sum_{m'\in\mathbb{Z}_M}\mathbf{1}_S(m')e^{-2\pi imm'/M}\bigg)p^{|S|}(1-p)^{M-|S|}\\
&=\frac{1}{2}\sum_{\substack{S\subseteq\mathbb{Z}_M\\S\neq\emptyset}}\sum_{m'\in S} e^{-2\pi imm'/M}p^{|S|}(1-p)^{M-|S|}\\
&=\frac{1}{2}\sum_{m'\in\mathbb{Z}_M}e^{-2\pi imm'/M}\sum_{S\supseteq\{m'\}}p^{|S|}(1-p)^{M-|S|}.
\end{align*}
From here, we apply the binomial theorem to get
\begin{equation*}
\sum_{S\supseteq\{m'\}}p^{|S|}(1-p)^{M-|S|}
=p\sum_{T\subseteq\mathbb{Z}_M\setminus\{m'\}}p^{|T|}(1-p)^{(M-1)-|T|}
=p,
\end{equation*}
and so $\mathbb{E}[Z]=0$ by the geometric sum formula (recall that $m\neq0$ by choice).
Next, we compute $\operatorname{Var}(Z)=\mathbb{E}[|Z-\mathbb{E}[Z]|^2]=\mathbb{E}[|Z|^2]=\mathbb{E}[Z\overline{Z}]$:
\begin{align*}
\operatorname{Var}(Z)
&=\sum_{S\subseteq\mathbb{Z}_M}\bigg(\frac{1}{2}\sum_{m'\in\mathbb{Z}_M}\mathbf{1}_S(m')e^{-2\pi imm'/M}\bigg)\bigg(\frac{1}{2}\sum_{m''\in\mathbb{Z}_M}\mathbf{1}_S(m'')e^{2\pi imm''/M}\bigg)p^{|S|}(1-p)^{M-|S|}\\
&=\frac{1}{4}\sum_{m'\in\mathbb{Z}_M}\sum_{m''\in\mathbb{Z}_M}e^{2\pi im(m''-m')}\sum_{S\supseteq\{m',m''\}}p^{|S|}(1-p)^{M-|S|}.
\end{align*}
At this point, we break the sum over $m''\in\mathbb{Z}_M$ into cases in which $m''=m'$ and $m''\neq m'$:
\begin{align*}
\operatorname{Var}(Z)
&=\frac{1}{4}\sum_{\substack{m'\in\mathbb{Z}_M\\m''=m'}}e^{2\pi im(m''-m')}\sum_{S\supseteq\{m',m''\}}p^{|S|}(1-p)^{M-|S|}\\
&\qquad+\frac{1}{4}\sum_{m'\in\mathbb{Z}_M}\sum_{\substack{m''\in\mathbb{Z}_M\\m''\neq m'}}e^{2\pi im(m''-m')}\sum_{S\supseteq\{m',m''\}}p^{|S|}(1-p)^{M-|S|}\\
&=\frac{1}{4}Mp+\frac{1}{4}\sum_{m'\in\mathbb{Z}_M}\sum_{\substack{m''\in\mathbb{Z}_M\\m''\neq m'}}e^{2\pi im(m''-m')}\sum_{S\supseteq\{m',m''\}}p^{|S|}(1-p)^{M-|S|}.
\end{align*}
Here, another application of the binomial theorem gives
\begin{equation*}
\sum_{S\supseteq\{m',m''\}}p^{|S|}(1-p)^{M-|S|}
=p^2\sum_{T\subseteq\mathbb{Z}_M\setminus\{m',m''\}}p^{|T|}(1-p)^{(M-2)-|T|}
=p^2,
\end{equation*}
and so by the geometric sum formula, we have
\begin{align*}
\operatorname{Var}(Z)
&=\frac{1}{4}Mp+\frac{1}{4}p^2\sum_{m'\in\mathbb{Z}_M}\sum_{\substack{m''\in\mathbb{Z}_M\\m''\neq m'}}e^{2\pi im(m''-m')}\\
&=\frac{1}{4}Mp+\frac{1}{4}p^2\sum_{m'\in\mathbb{Z}_M}(-1)
=\frac{1}{4}Mp-\frac{1}{4}Mp^2
=\frac{1}{4}Mp(1-p).
\end{align*}
Having calculated the expected value and variance of $Z$, we are now ready to use Lemma~\ref{lemma.chernoff bound} (and the fact that $Z=\frac{M}{2}(F^*\mathbf{1}_B)(m)$).
A union bound gives
\begin{align*}
\operatorname{Pr}\bigg(\|B\|_u\geq\frac{2t\sigma}{M}\bigg)
&\leq(M-1)\operatorname{Pr}\bigg(|(F^*\mathbf{1}_B)(m)|\geq\frac{2t\sigma}{M}\bigg)\\
&=(M-1)\operatorname{Pr}\Big(\big|Z-\mathbb{E}[Z]\big|\geq t\sigma\Big)
\leq4(M-1)\max\Big\{e^{-t^2/8},e^{-t\sigma/2\sqrt{2}}\Big\}.
\end{align*}
We select $t=(\sqrt{2c(1+\varepsilon)}\log M)/\sigma$ and simplify the exponents in our probability bound:
\begin{equation*}
\frac{t^2}{8}
=\frac{c(1+\varepsilon)\log^2 M}{Mp(1-p)}
\geq\frac{c(1+\varepsilon)\log^2 M}{Mp}
=(1+\varepsilon)\log M,
\qquad
\frac{t\sigma}{2\sqrt{2}}
=\frac{\sqrt{c(1+\varepsilon)}}{2}\log M.
\end{equation*}
With this choice of $t$, we continue:
\begin{align*}
\operatorname{Pr}\bigg(\|B\|_u\geq\sqrt{8c(1+\varepsilon)}\cdot\frac{\log M}{M}\bigg)
&\leq 4(M-1)\max\Big\{e^{-(1+\varepsilon)\log M},e^{-(\sqrt{c(1+\varepsilon)}/2)\log M}\Big\}\\
&\leq 4\max\Big\{M^{-\varepsilon},M^{1-\sqrt{c(1+\varepsilon)}/2}\Big\}.
\end{align*}
Since $c\geq4$, we therefore have that $\|B\|_u\leq 3\sqrt{c}\cdot\frac{\log M}{M}$ with high probability.
\end{proof}

We can now combine Lemmas~\ref{lemma.no collisions} and~\ref{lemma.small fourier bias} to produce a graph $G$ of the form in Definition~\ref{definition.graph construction} with a large spectral gap:

\begin{theorem}
\label{theorem.bound on spectral gap}
Pick $c\geq4$ and suppose the entries of $\mathbf{1}_B$ are independent, identical Bernoulli random variables with mean $\frac{c\log M}{M}$.
Take $A:=B\cup(-B)\setminus\{0\}$ and define $G$ according to Definition~\ref{definition.graph construction}.
Then
\begin{equation*}
\operatorname{spg}(G)
\geq1-\frac{6}{\sqrt{c}}
\end{equation*}
with high probability.
\end{theorem}

\begin{proof}
We will assume that the events of Lemmas~\ref{lemma.no collisions} and~\ref{lemma.small fourier bias} hold simultaneously, as they will with high probability.
Then $|A|=2|B|$, and furthermore by \eqref{eq.bound on fourier bias}, $\|A\|_u\leq2\|B\|_u$.
Starting with Lemma~\ref{lemma.spectral gap vs fourier bias}, we then have
\begin{equation*}
\operatorname{spg}(G)
=1-\frac{M}{|A|}\|A\|_u
\geq1-\frac{M}{|B|}\|B\|_u
\geq1-\frac{M}{\frac{1}{2}c\log M}\cdot3\sqrt{c}\cdot\frac{\log M}{M}
=1-\frac{6}{\sqrt{c}},
\end{equation*}
where the second inequality applies Lemmas~\ref{lemma.no collisions}(iii) and~\ref{lemma.small fourier bias}.
\end{proof}

We now apply Lemma~\ref{lemma.spectral gap vs connectivity} to show that $\mathcal{O}(\log M)$ Fourier masks suffice for injectivity:

\begin{theorem}
\label{theorem.main result}
Take $K=12$, $c=144$, suppose the entries of $\mathbf{1}_B$ are independent, identical Bernoulli random variables with mean $\frac{c\log M}{M}$, and take $A:=B\cup(-B)\setminus\{0\}$.
Then with high probability, the $\leq 2\cdot10^5\cdot\log M$ masks 
\begin{equation*}
\{D_k\}_{k=0}^{K-1}\cup\{D_k+\omega^rE^aD_{k'}\}_{k=0,}^{K-1}\ _{k'=0,}^{k}\ _{r=0,}^{2}\ _{a\in A}^{},
\end{equation*}
as described in Lemmas~\ref{lemma.full spark phi_v} and~\ref{lemma.modulation trick}, lend injective intensity measurements.
\end{theorem}

\begin{proof}
Note that $G$ has $n=KM$ vertices and is $d$-regular with $d=K|A|$.
We need to be robust to the removal of any $M-1$ vertices, which in turn removes $d(M-1)$ edges, and so it suffices to be robust to the removal of any $dM$ edges.
As such, we observe Lemma~\ref{lemma.spectral gap vs connectivity} and take $\varepsilon$ to satisfy
\begin{equation}
\label{eq.first requirement on spg}
\frac{1}{K}
=\frac{dM}{dn}
=\varepsilon
\leq\frac{\operatorname{spg}(G)}{6}.
\end{equation}
We also want a connected component of size $M$ after the removal of these edges, and so by Lemma~\ref{lemma.spectral gap vs connectivity}, it suffices to have
\begin{equation}
\label{eq.second requirement on spg}
M
\leq\bigg(1-\frac{2\varepsilon}{\operatorname{spg}(G)}\bigg)n
=\bigg(1-\frac{2}{K\operatorname{spg}(G)}\bigg)KM.
\end{equation}
Rearranging \eqref{eq.first requirement on spg} and \eqref{eq.second requirement on spg} produces the following specification on the spectral gap of $G$ for sufficient connectivity:
\begin{equation*}
\operatorname{spg}(G)
\geq\max\bigg\{\frac{6}{K},\frac{2}{K-1}\bigg\}
=\frac{6}{K},
\end{equation*}
where the equality is valid provided $K\geq2$.
To meet this specification, based on Theorem~\ref{theorem.bound on spectral gap}, it suffices to have $1-6/\sqrt{c}\geq 6/K$, which requires $K>6$ and is equivalent to having $c\geq(1/6-1/K)^{-2}$.
Using Lemma~\ref{lemma.no collisions}, the total number of masks is
\begin{equation*}
K+3\binom{K+1}{2}|A|
=K+6\binom{K+1}{2}|B|
\leq K+9\binom{K+1}{2}c\log M.
\end{equation*}
Setting $c=(1/6-1/K)^{-2}$, then $K=12$ minimizes the coefficient of $\log M$.
\end{proof}

At this point, we note that $2\cdot10^5\cdot\log M$ is a rather large number of Fourier masks.
Certainly, the $10^5$ might be an artifact of our analysis---perhaps it could be decreased by leveraging better approximations.
However, as the next result shows, the log factor is necessary for the Fourier masks to lend polarization-based recovery:

\begin{theorem}
\label{theorem.log is necessary}
Take $G$ as defined in Definition~\ref{definition.graph construction}.
Then $\operatorname{spg}(G)>\varepsilon$ only if 
\begin{equation*}
|A|
\geq\frac{\log M}{2+\log(1/\varepsilon)}.
\end{equation*}
\end{theorem}

\begin{proof}
Define $V$ to be the $|A|\times M$ matrix built from taking the rows of $F$ indexed by $A$ and then scaling the columns to have unit norm.
Then the inner product between any two columns of $V$ is given by
\begin{align*}
\langle v_m,v_{m'}\rangle
&=\sum_{a\in A}\bigg(\frac{1}{\sqrt{|A|}}e^{2\pi ima/M}\bigg)\bigg(\frac{1}{\sqrt{|A|}}e^{-2\pi im'a/M}\bigg)\\
&=\frac{M}{|A|}\cdot\frac{1}{M}\sum_{m''\in\mathbb{Z}_M}\mathbf{1}_A(m'')e^{-2\pi i(m'-m)m''/M}
=\frac{M}{|A|}(F^*\mathbf{1}_A)(m'-m).
\end{align*}
As such, the worst-case coherence between columns of $V$ can be expressed in terms of the Fourier bias of $A$ (and the spectral gap of $G$ by Lemma~\ref{lemma.spectral gap vs fourier bias}):
\begin{equation}
\label{eq.coherence to fourier bias}
\max_{\substack{m,m'\in\mathbb{Z}_M\\m\neq m'}}|\langle v_m,v_{m'}\rangle|
=\frac{M}{|A|}\|A\|_u
=1-\operatorname{spg}(G)
\leq 1-\varepsilon.
\end{equation}
As you might expect, unit vectors can only be incoherent if there are enough dimensions relative to the number of vectors.
To establish this explicitly, view the $v_m$'s as vectors in $\mathbb{R}^{2|A|}$ by stacking the real and imaginary parts of each entry.
Letting $\delta$ denote the shortest distance between distinct $v_m$'s (this is the same in $\mathbb{C}^{|A|}$ and $\mathbb{R}^{2|A|}$), consider the open balls of radius $\delta/2$ centered at each $v_m$.
The disjoint union of these $M$ balls is contained in the ball of radius $1+\delta/2$ centered at the origin, and so a volume comparison gives
\begin{equation*}
M\cdot C\big(\tfrac{\delta}{2}\big)^{2|A|}
=\operatorname{Vol}\bigg(\bigsqcup_{m\in\mathbb{Z}_M}B\big(v_m,\tfrac{\delta}{2}\big)\bigg)
\leq\operatorname{Vol}\bigg(B\big(0,1+\tfrac{\delta}{2}\big)\bigg)
=C\big(1+\tfrac{\delta}{2}\big)^{2|A|}.
\end{equation*}
Rearranging this inequality then yields
\begin{equation}
\label{eq.bound on size of a}
|A|
\geq\frac{\log M}{2\log(2/\delta+1)}.
\end{equation}
To analyze $\delta$, note that for any $m,m'\in\mathbb{Z}_M$, we have
\begin{equation*}
\|v_m-v_{m'}\|^2
=\|v_m\|^2-2\operatorname{Re}\langle v_m,v_{m'}\rangle+\|v_{m'}\|^2
\geq2-2|\langle v_m,v_{m'}\rangle|
\geq2\varepsilon,
\end{equation*}
where the last step is by \eqref{eq.coherence to fourier bias}.
Thus $\delta^2\geq 2\varepsilon$, with which we continue \eqref{eq.bound on size of a}:
\begin{equation*}
|A|
\geq\frac{\log M}{2\log(\sqrt{2/\varepsilon}+1)}
\geq\frac{\log M}{2+\log(1/\varepsilon)}.
\end{equation*}
Here, the last inequality can be verified using the fact that $\varepsilon\leq1$.
\end{proof}

\section{Numerical simulations}

In this section, we compare the polarization method of this paper to a state-of-the-art phase retrieval algorithm known as \textit{PhaseLift}~\cite{CandesESV:11,CandesSV:11}.
The main idea behind PhaseLift is that the intensity measurement $x\mapsto|\langle x,\varphi\rangle|^2$ can be viewed as a linear measurement if the signal $x$ is ``lifted'' to the outer product $xx^*$ in the real vector space of self-adjoint $M\times M$ matrices.
Indeed, the intensity measurement is a Hilbert-Schmidt inner product in this vector space: 
\begin{equation*}
|\langle x,\varphi\rangle|^2
=\varphi^*xx^*\varphi
=\operatorname{Tr}[\varphi^*xx^*\varphi]
=\operatorname{Tr}[\varphi\varphi^*xx^*]
=\langle xx^*,\varphi\varphi^*\rangle_{\mathrm{HS}}.
\end{equation*}
However, this larger vector space has $M^2$ dimensions, and so $M^2$ inner products are necessary to identify any member of the space---that is, unless more information is available.
In this case, we know that the desired self-adjoint matrix $xx^*$ is positive semidefinite with rank $1$, and so one could seek to minimize rank over the positive semidefinite matrices with the given intensity measurements.
Rank minimization tends to be a difficult program to solve, so one is inclined to relax it: 
\begin{equation*}
\mbox{(PhaseLift)}
\qquad\qquad
\mbox{minimize}
\quad
\operatorname{Tr}(X)
\quad
\mbox{s.t.}
\quad
\mathcal{A}(X)=b,~X\succeq0.
\end{equation*}

To date, PhaseLift is known to stably reconstruct all signals in $\mathbb{C}^M$ with only $\mathcal{O}(M)$ complex Gaussian measurements~\cite{CandesL:12}.
However, very little is known about its performance when the measurement vectors do not come from a unitarily invariant distribution (e.g., in the Fourier masks setting).
To be fair, the present paper also does not prove stable reconstruction from Fourier masks; indeed, the previous section ``merely'' established injectivity with efficient recoverability.
Therefore, for the sake of completeness, this section will provide numerical simulations that compare the stability of both phase retrieval algorithms.

First, we describe how we implement phase retrieval with polarization.
Starting with an ensemble of vertex measurement vectors $\Phi_V$ and edge measurement vectors $\Phi_E$, then for each edge $(i,j)\in E$, we apply \eqref{eq.polarization trick} to get
\begin{equation*}
w_{ij}
=\frac{1}{3}\sum_{r=0}^2\omega^r\big(|\langle x,\varphi_i+\omega^r\varphi_j\rangle|^2+\varepsilon_{ijr}\big),
\end{equation*}
where $\varepsilon_{ijr}$ denotes the noise added to the $(i,j,r)$th edge measurement.
Later, we will normalize $w_{ij}$ in order to estimate the relative phase between the measurements at vertices $i$ and $j$, but if $w_{ij}$ is small, this normalization will be particularly susceptible to noise.
As such, we first remove vertices which are adjacent to small edge weights so as to promote reliability in the edges.

\begin{algorithm}[h]
\caption{Pruning for reliability}
\SetAlgoLined
\KwIn{Weighted graph $G=(V,E,W)$, pruning parameter $\alpha$}
\KwOut{Subgraph $G'$ with a larger smallest edge weight}
Initialize $G'\leftarrow G$\\
\For{$i=1$ \KwTo $\lfloor(1-\alpha)|V|\rfloor$}{
Find edge $(i,j)\in E$ of minimal weight\\
$G'\leftarrow G'\setminus\{i,j\}$
}
\end{algorithm}

Now that we have isolated a subgraph with reliable edges, we want to find a further subgraph which is sufficiently connected so that its vertices can democratically agree on the relative phases.
To find this subgraph, we iteratively remove vertices implicated by \textit{spectral clustering}~\cite{Alon:86,AlonM:85,Cheeger:70} until the spectral gap is sufficiently large.

\begin{algorithm}[h]
\caption{Pruning for connectivity}
\SetAlgoLined
\KwIn{Graph $G'=(V',E')$, pruning parameter $\tau$}
\KwOut{Subgraph $G''$ with $\operatorname{spg}(G'')\geq\tau$}
Initialize $G''\leftarrow G'$\\
\While{$\operatorname{spg}(G'')<\tau$}{
Take $D$ to be the diagonal matrix of vertex degrees\\
Compute the Laplacian $L\leftarrow I-D^{-1/2}AD^{-1/2}$\\
Compute the eigenvector $u$ corresponding to the second eigenvalue of $L$\\
\For{$i=1$ \KwTo $|V''|$}{
Let $S_i$ denote the vertices corresponding to the $i$ smallest entries of $D^{-1/2}u$\\
Let $E(S_i,S_i^\mathrm{c})$ denote the number of edges between $S_i$ and $S_i^\mathrm{c}$\\
$h_i\leftarrow E(S_i,S_i^\mathrm{c})/\min\{\sum_{v\in S_i}\operatorname{deg}(v),\sum_{v\in S_i^\mathrm{c}}\operatorname{deg}(v)\}$
}
Take $S$ to be the $S_i$ of minimal $h_i$\\$G''\leftarrow G''\setminus S$
}
\end{algorithm}

At this point, our graph has reliable edges and is well connected.
We now run \textit{angular synchronization}~\cite{BandeiraSS:12,Singer:11} to reach a consensus on the phases of the vertex measurements, up to a global phase factor.

\begin{algorithm}[h]
\caption{Angular synchronization}
\SetAlgoLined
\KwIn{Graph $G''=(V'',E'',W'')$}
\KwOut{Vector of phases corresponding to vertex measurements}
Let $A_1$ denote the weighted adjacency matrix with the weights normalized\\
Let $D$ denote the diagonal matrix of vertex degrees\\
Compute the connection Laplacian $L_1\leftarrow I-D^{-1/2}A_1D^{-1/2}$\\
Compute the eigenvector $u$ corresponding to the smallest eigenvalue of $L_1$\\
Output the phases of the coordinates of $u$
\end{algorithm}

Having estimated the phases of the vertex measurements corresponding to $V''$, we now multiply the square roots of the vertex measurements $\{\sqrt{|\langle x,\varphi_i\rangle|^2+\varepsilon_{i}}\}_{i\in V''}$ by these phases to estimate the inner products $\{\langle x,\varphi_i\rangle\}_{i\in V''}$, and then produce a least-squares estimate for the desired signal $x$.
As established in~\cite{AlexeevBFM:12}, this implementation of the polarization method is stable to noise when the vertex measurement vectors are complex Gaussian.
In this section, we run numerical simulations to illustrate stability in the Fourier masks setting.

In the experiments that follow, we run the above implementation of polarization using pruning parameters $\alpha=0.99$, $\tau=0.1$, and the following construction of masks:
First, we draw $K=3$ masks whose diagonal entries are independent with distribution $\mathcal{N}(0,1)$.
Next, we draw $M-1$ independent realizations of a random variable $X$ with uniform distribution over the interval $[0,1]$, and we put $i\in B$ whenever $X_i\leq\frac{\log M}{M}$.
This way, $B\subseteq\mathbb{Z}_M$ is a subset that does not contain $0$, but the nonzero members of $\mathbb{Z}_M$ each reside in $B$ with probability $\frac{\log M}{M}$, thereby following the intent of Theorem~\ref{theorem.main result}.
Since $0\not\in B$, we take $A=B\cup(-B)$, and then we construct the auxiliary masks $\{D_k+\omega^rE^aD_{k'}\}_{r=0}^2$ for every $(k,k')\in\{0,1,2\}$, $k\geq k'$ and $a\in A$.
Overall, in these experiments, polarization will use a total of $18|A|+3$ masks, where $|A|$ tends to be approximately $2\log M$.

By contrast, PhaseLift offers a lot more flexibility with the number and types of masks used, and this flexibility allows us to run a collection of choice comparisons.
We will run experiments at three noise levels, and for each level, we will compare polarization to three different implementations of PhaseLift.
In the first implementation, we will only give PhaseLift the measurements corresponding to the original $3$ masks we give to polarization.
Next, we will give PhaseLift all $18|A|+3$ of the masks we described in the previous paragraph.
For the last comparison, we will give PhaseLift the original $3$ masks along with $18|A|$ additional masks whose diagonal entries are also independent with distribution $\mathcal{N}(0,1)$. 
Based on discussions in~\cite{CandesESV:11}, this last setup appears to be the intended design of Fourier masks for PhaseLift, so in a sense, this last comparison will allow both algorithms to compete with their own ``home-field advantage.''

In the comparison between polarization and PhaseLift, we use two performance metrics:
run time and relative error of reconstruction, defined in this setting by
\begin{equation*}
\min_{\substack{c\in\mathbb{C}\\|c|=1}}\frac{\|c\hat{x}-x\|_2}{\|x\|_2}.
\end{equation*}
Here, $c$ is playing the role of the global phase factor we lost in the intensity measurement process.
For each noise level tested $\sigma^2\in\{0,0.1,1\}$, and for each type of mask ensemble given to PhaseLift (original $3$, same $18|A|+3$, random $18|A|+3$), we considered multiple signal dimensions $M\in\{2^5,2^6,2^7,2^8,2^9\}$.
In each of these scenarios, we ran $30$ realizations of the following experiment:
\begin{itemize}
\item Draw each entry of the signal $x$ independently from $\mathcal{N}(0,1)$.
\item Add independent $\mathcal{N}(0,\sigma^2)$ noise to each intensity measurement $|\langle x,\varphi\rangle|^2$.
\item Run both polarization and PhaseLift to estimate $x$ and record the run time and relative error.
\end{itemize}
We ran these experiments on an Intel\textregistered\ Core\texttrademark 2 Quad CPU Q9550 at 2.83GHz with 3.6 GB of memory.

\begin{figure}[p!]
\centering
\includegraphics[width=0.45\textwidth]{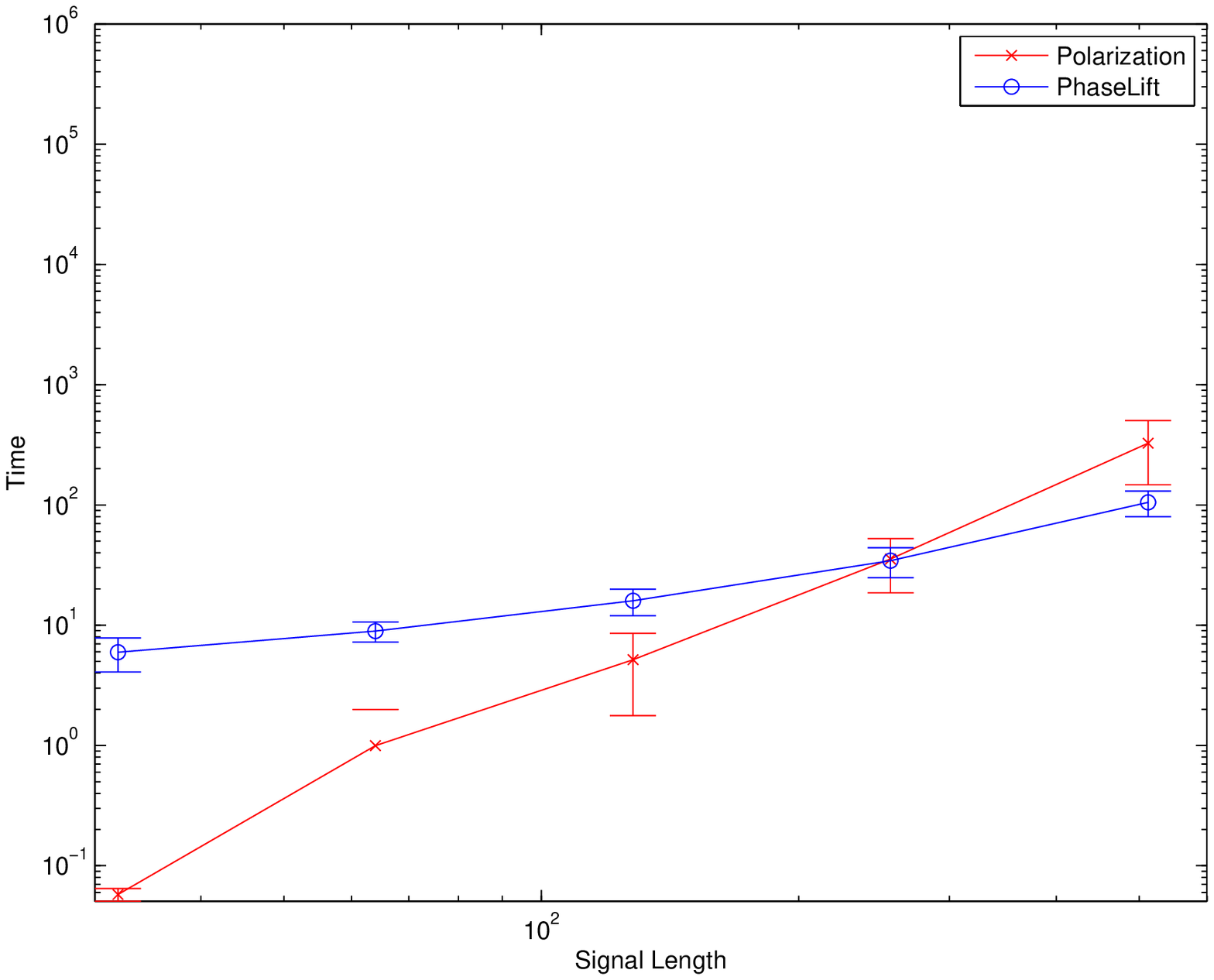}
\includegraphics[width=0.45\textwidth]{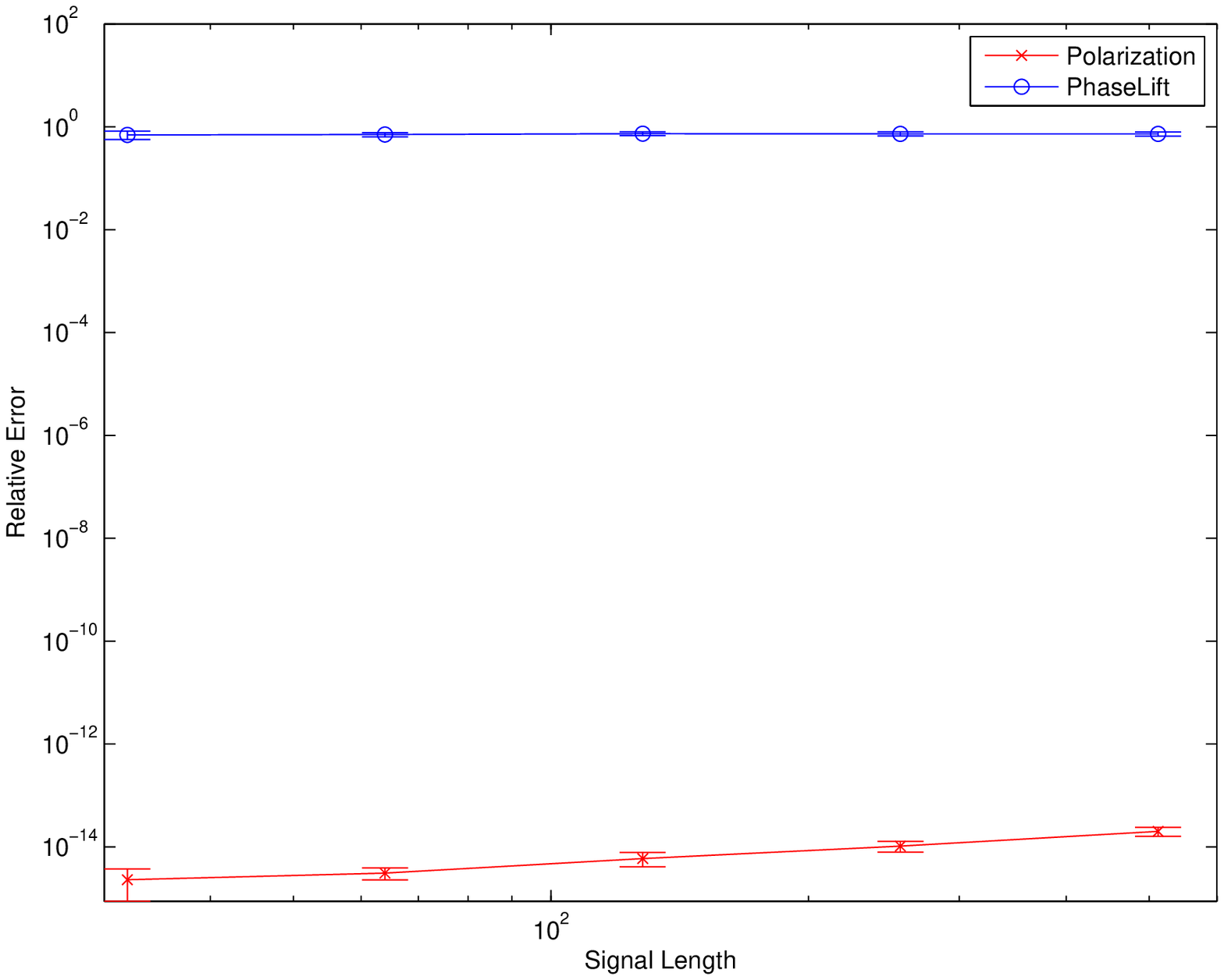}
\includegraphics[width=0.45\textwidth]{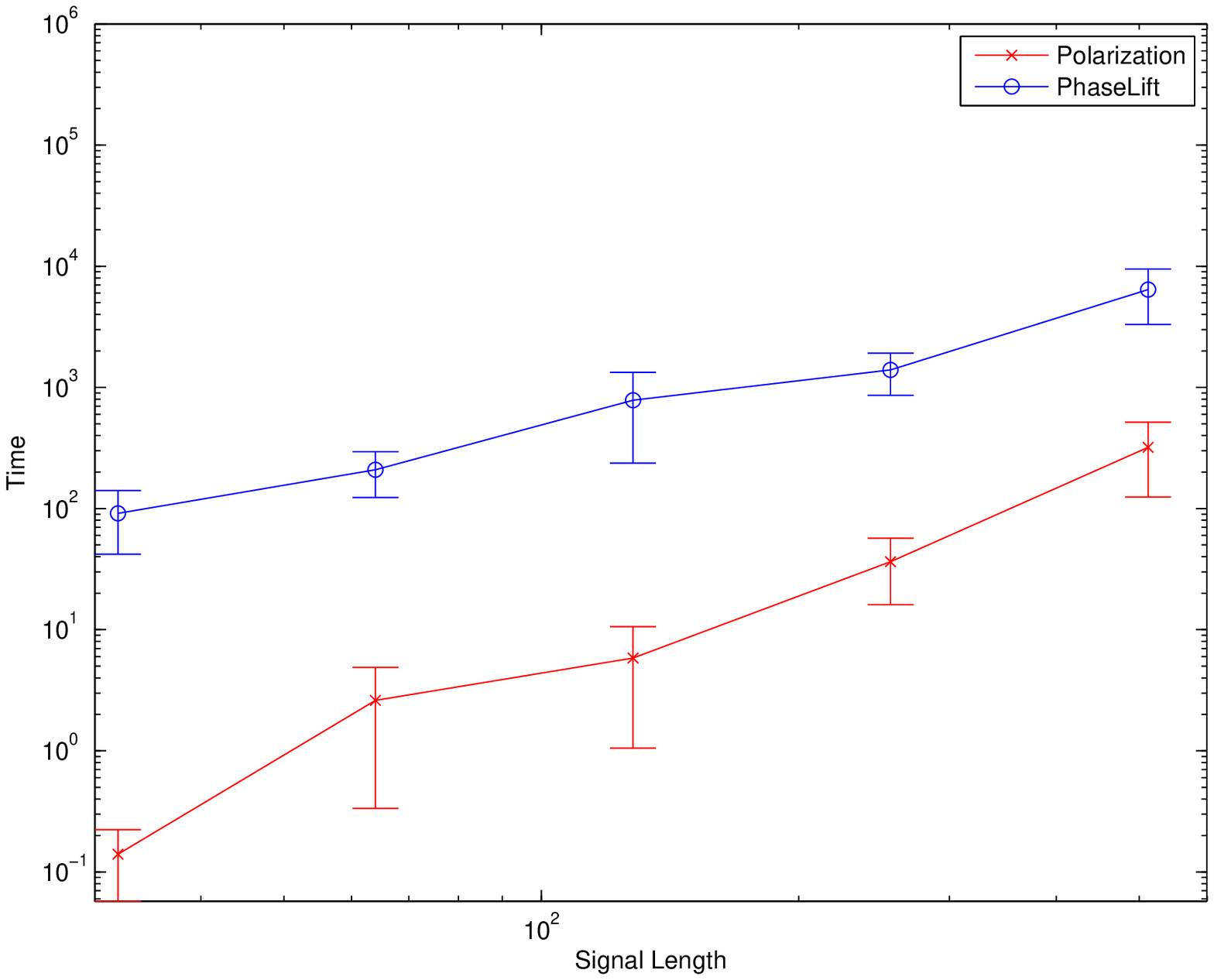}
\includegraphics[width=0.45\textwidth]{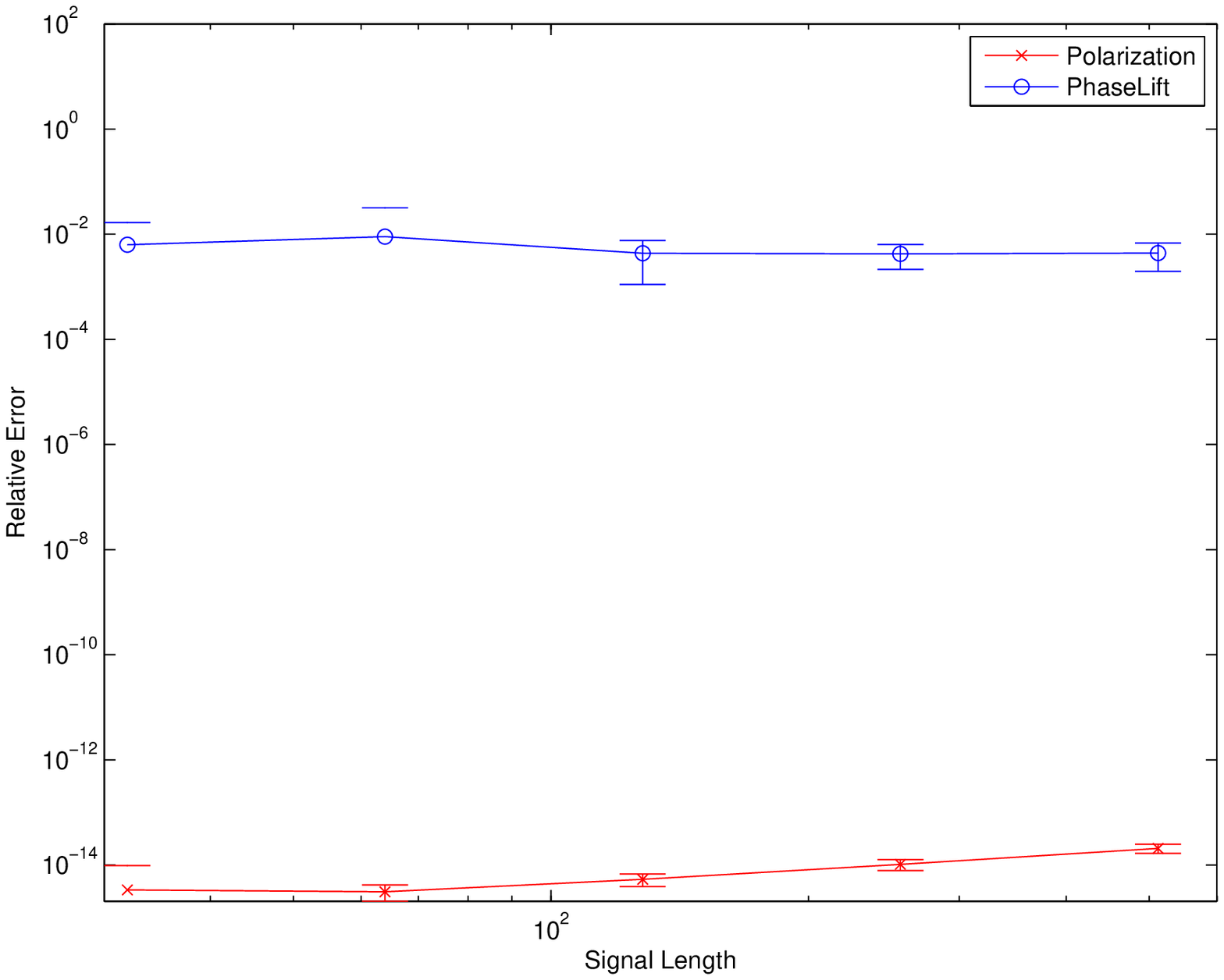}
\includegraphics[width=0.45\textwidth]{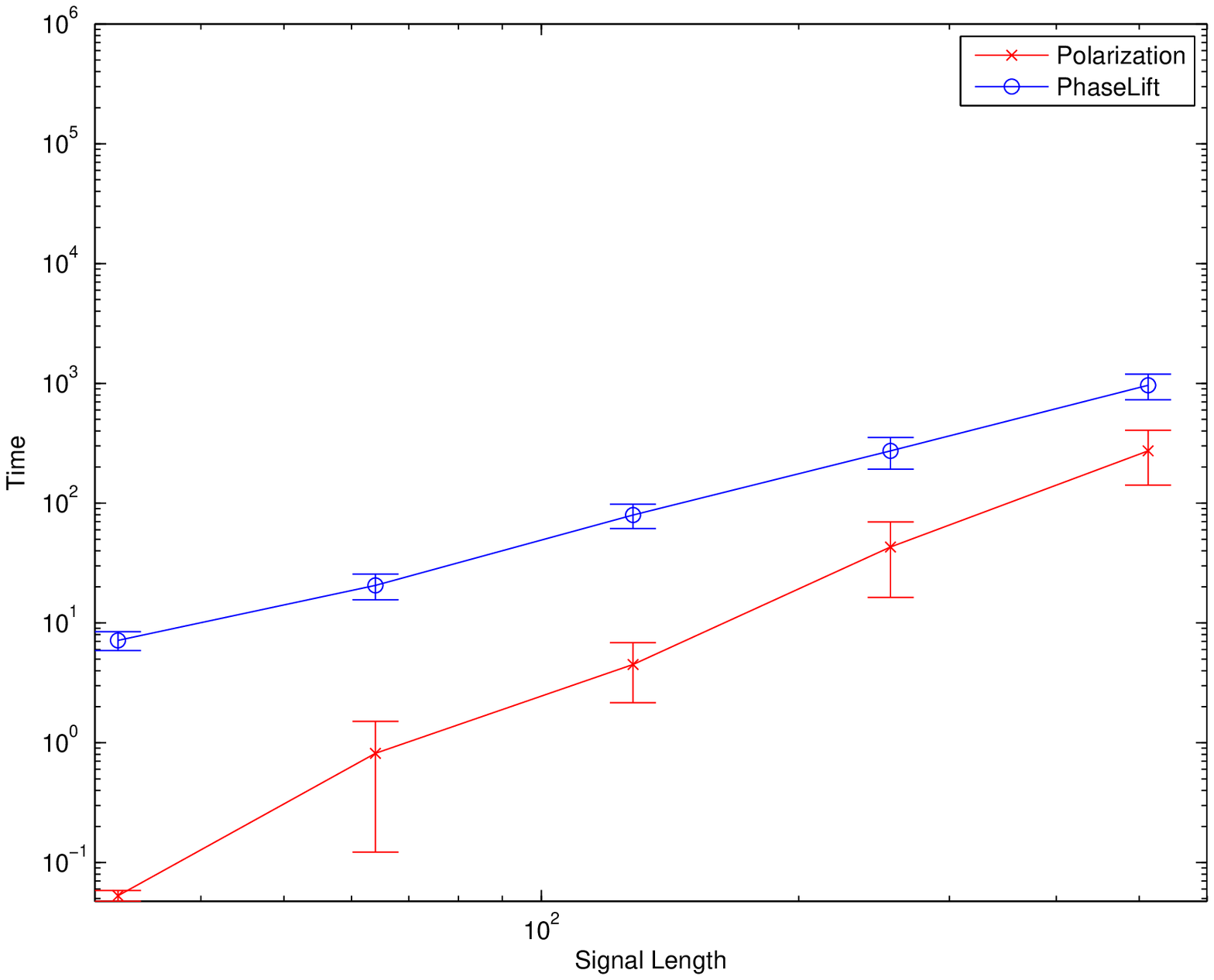}
\includegraphics[width=0.45\textwidth]{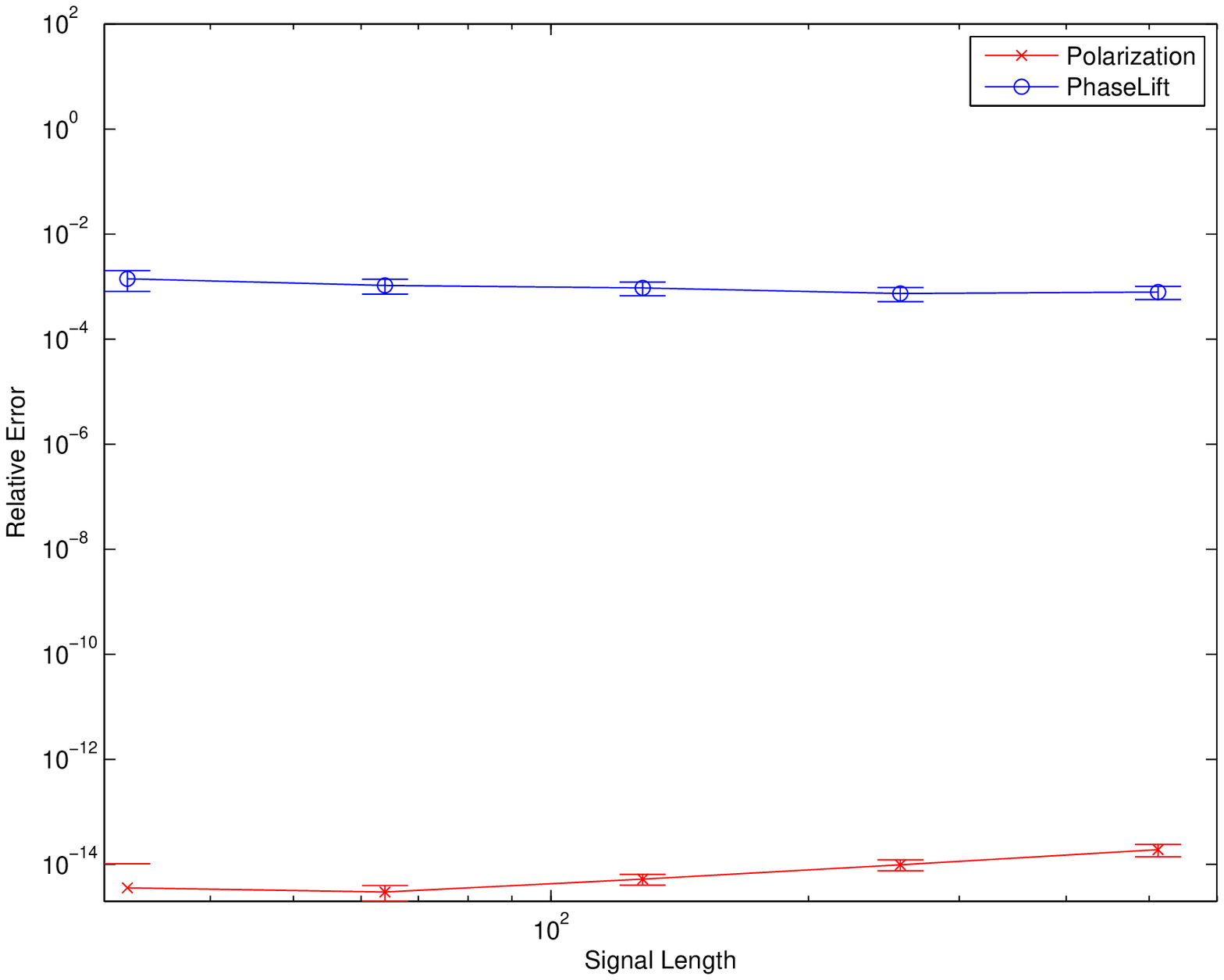}
\caption{
\label{figure.noise00}
Phase retrieval without noise.
For the first row of graphs, PhaseLift is given the original $K=3$ random masks, for the second, it is given all $18|A|+3$ of the masks given to polarization, and in the last row, PhaseLift is given $18|A|+3$ random masks.
The plotted data illustrate the sample mean plus/minus one sample standard deviation.
In some cases, the lower error bar is negative, and so it is not plotted in the log scale.
}
\end{figure}

\begin{figure}[p!]
\centering
\includegraphics[width=0.45\textwidth]{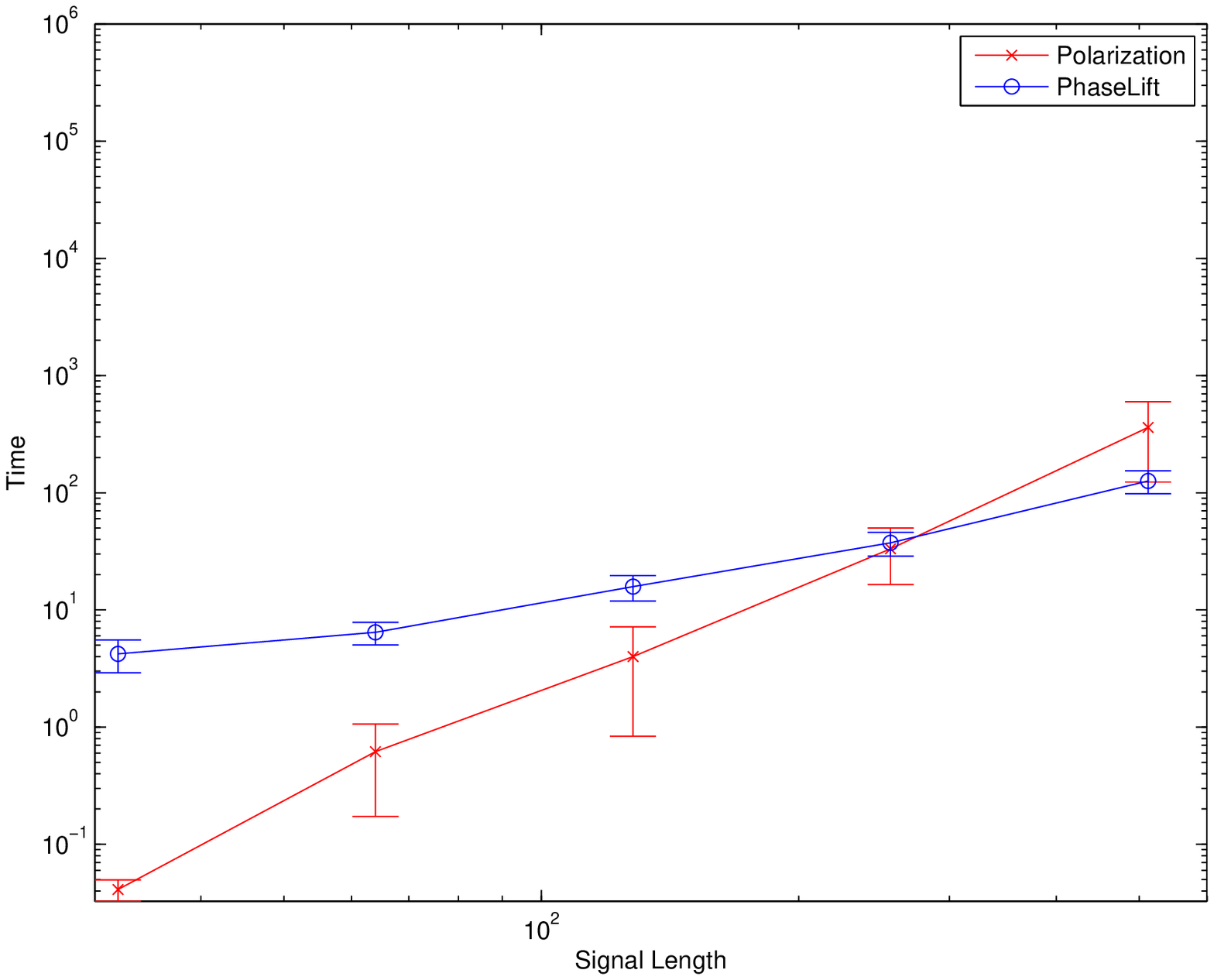}
\includegraphics[width=0.45\textwidth]{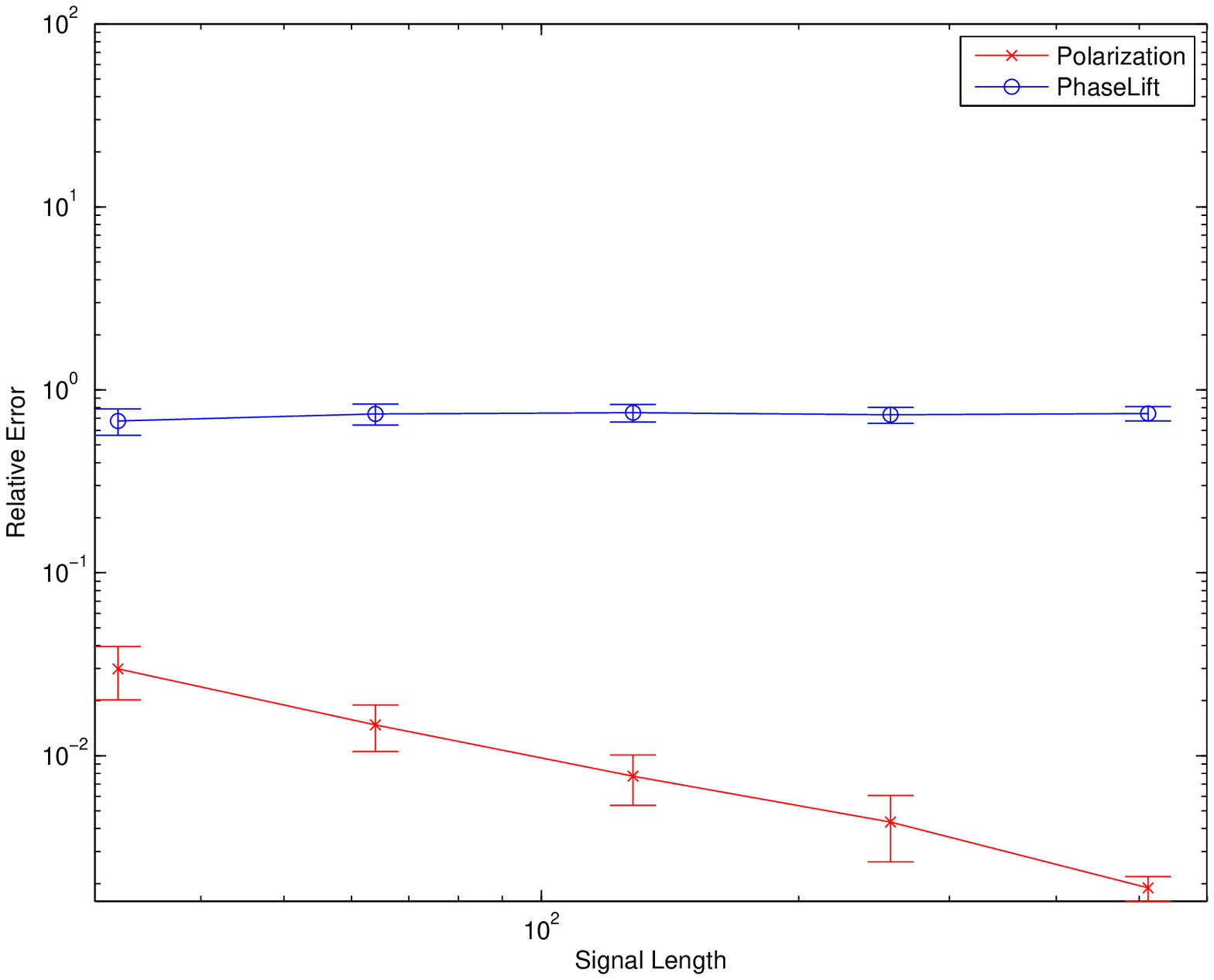}
\includegraphics[width=0.45\textwidth]{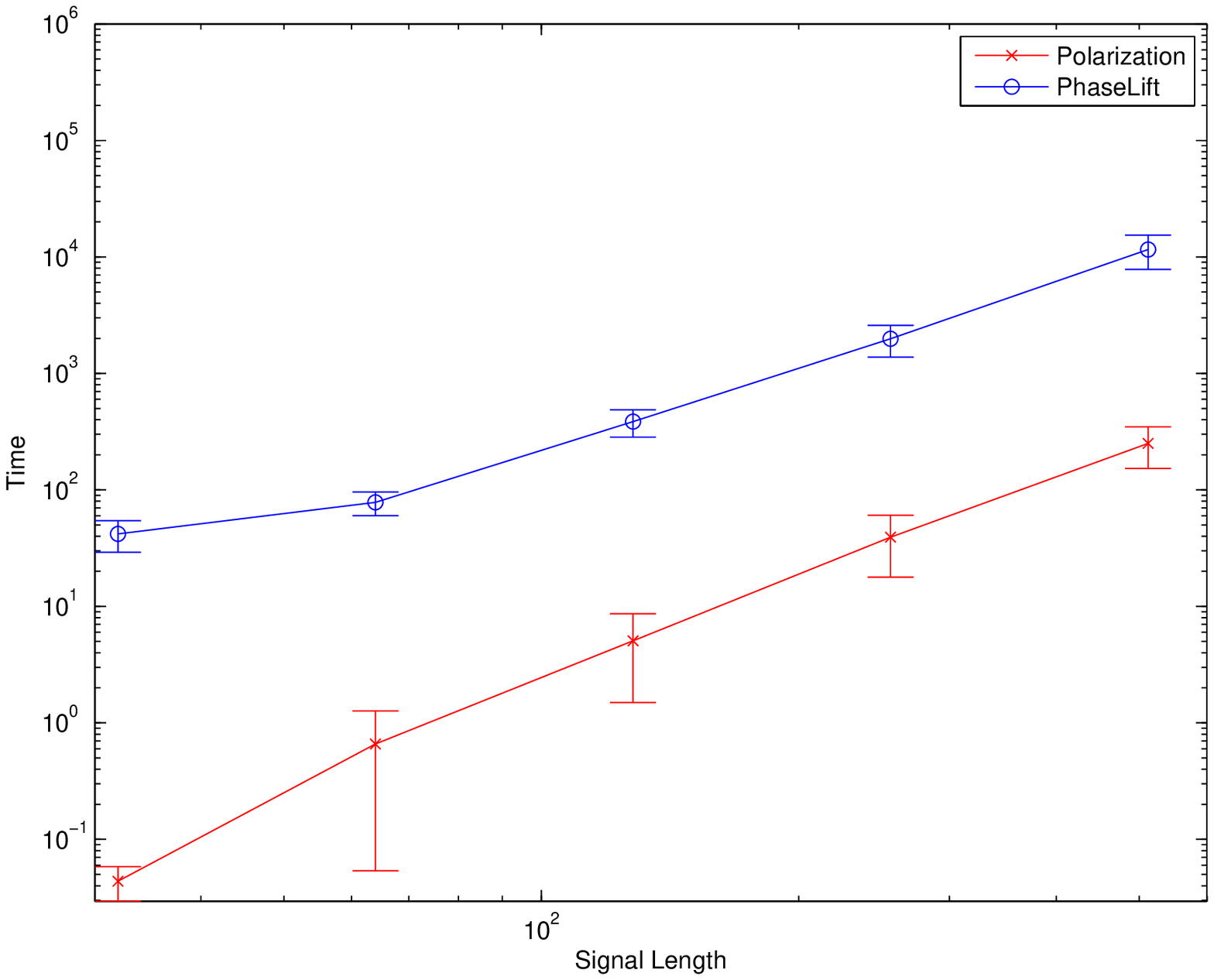}
\includegraphics[width=0.45\textwidth]{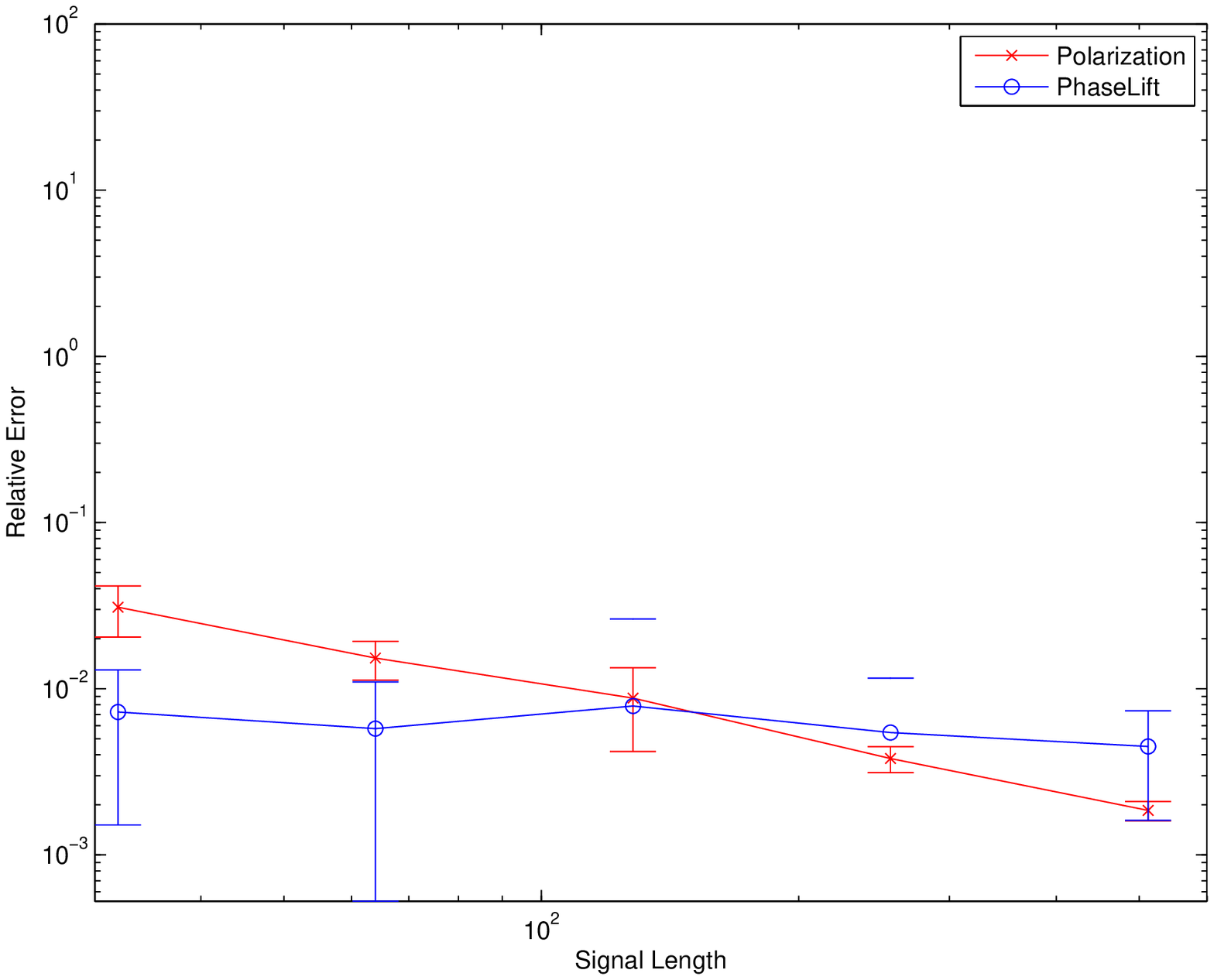}
\includegraphics[width=0.45\textwidth]{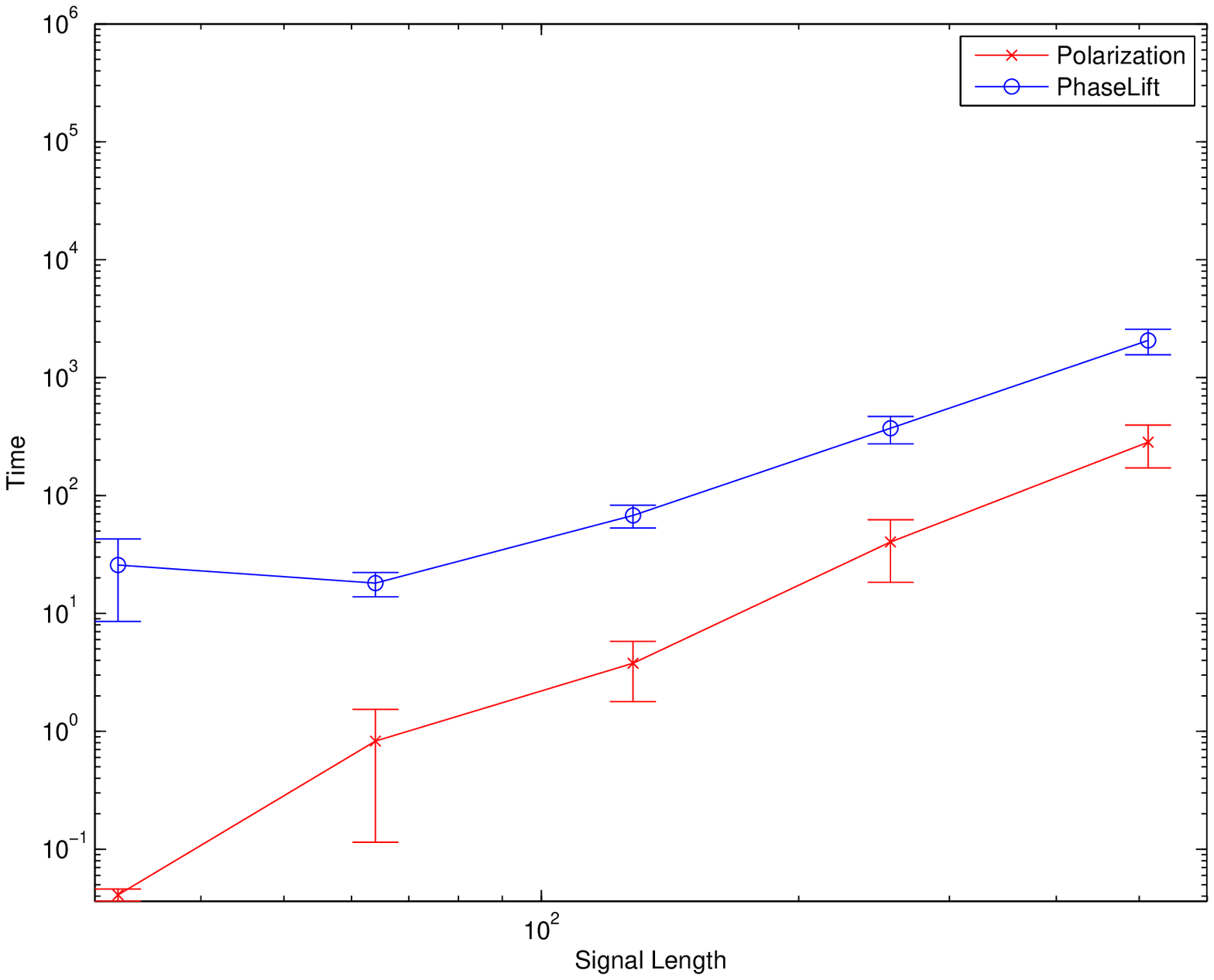}
\includegraphics[width=0.45\textwidth]{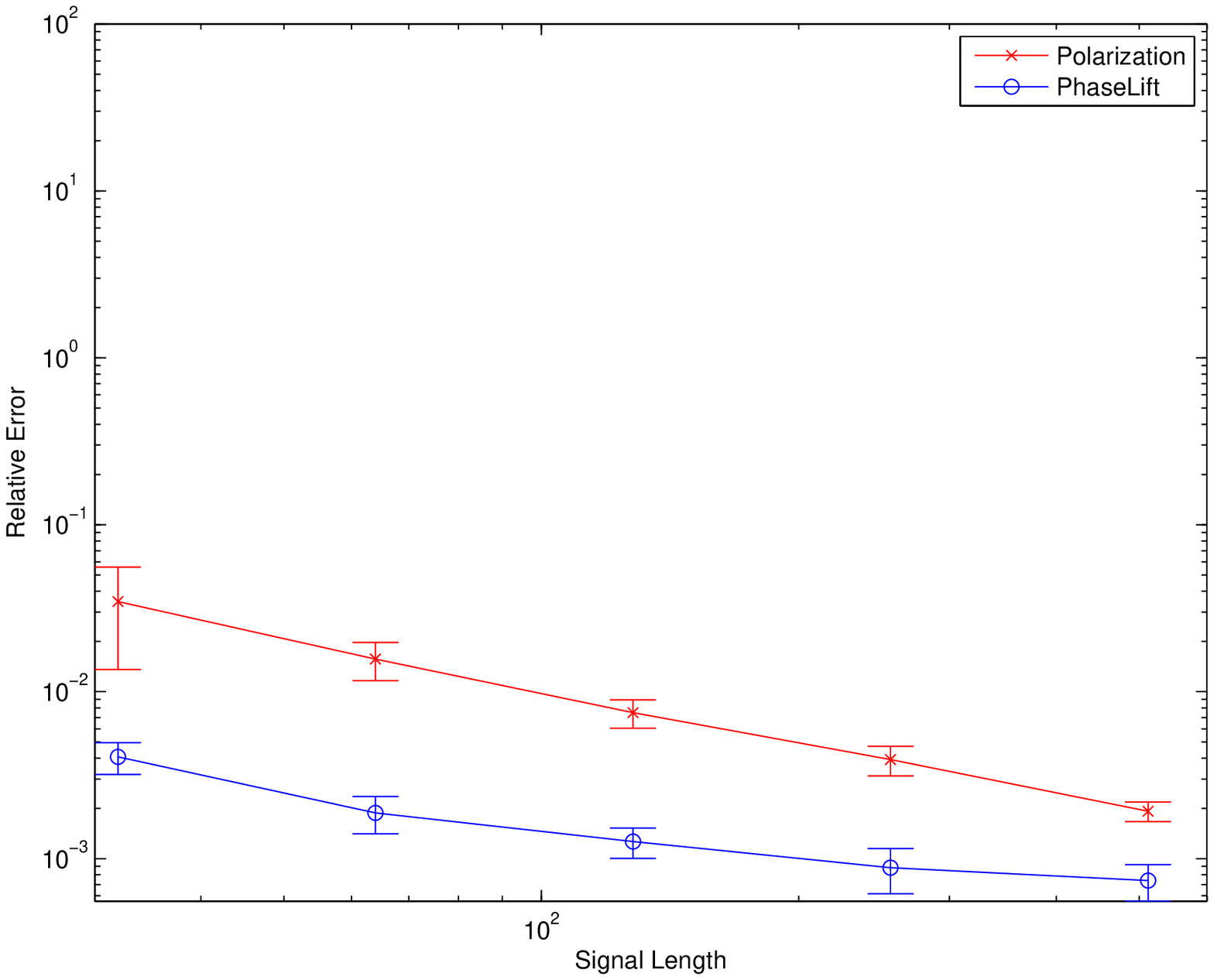}
\caption{
\label{figure.noise01}
Phase retrieval with $\mathcal{N}(0,0.1)$ noise added to each intensity measurement.
For the first row of graphs, PhaseLift is given the original $K=3$ random masks, for the second, it is given all $18|A|+3$ of the masks given to polarization, and in the last row, PhaseLift is given $18|A|+3$ random masks.
The plotted data illustrate the sample mean plus/minus one sample standard deviation.
In some cases, the lower error bar is negative, and so it is not plotted in the log scale.
}
\end{figure}

\begin{figure}[p!]
\centering
\includegraphics[width=0.45\textwidth]{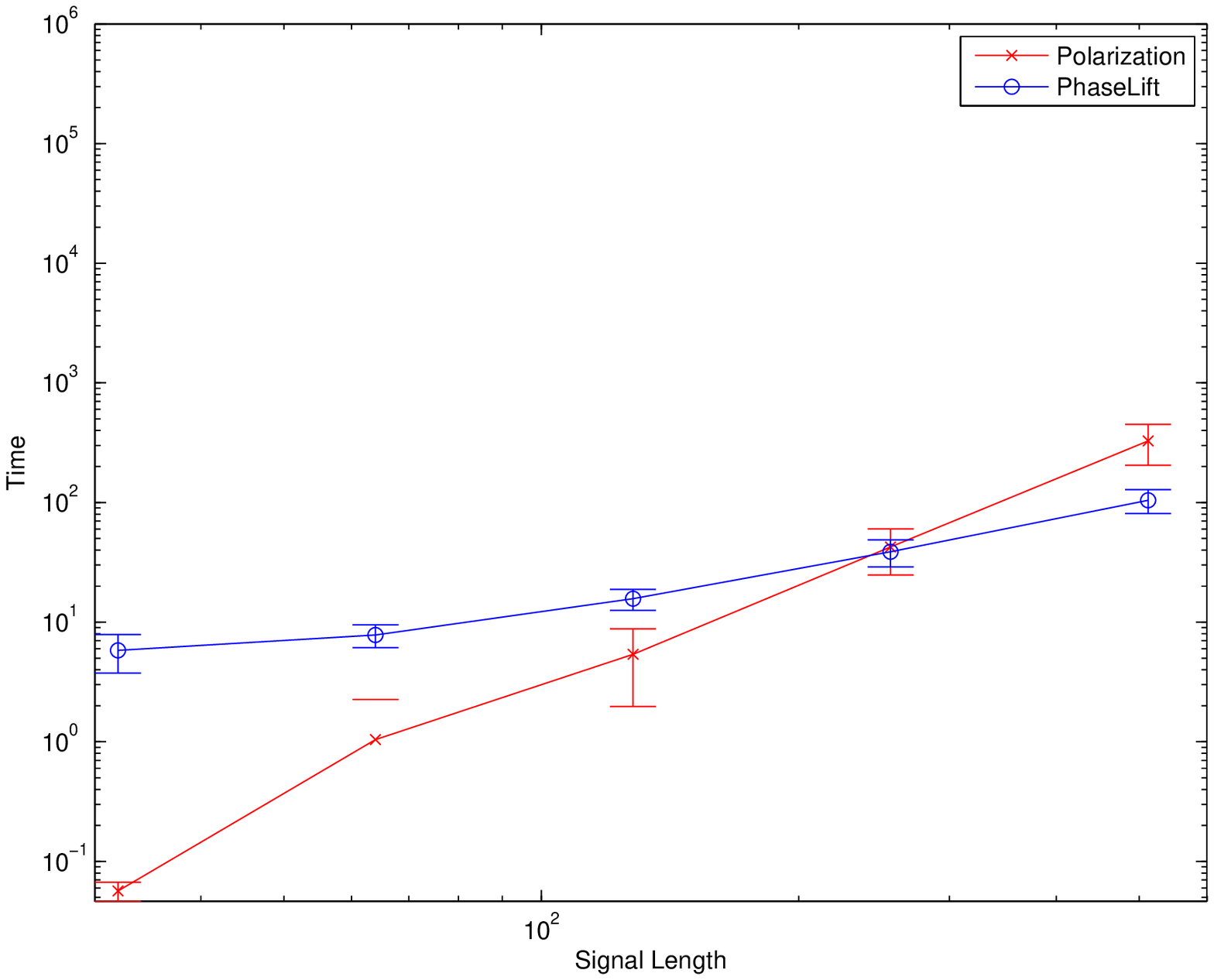}
\includegraphics[width=0.45\textwidth]{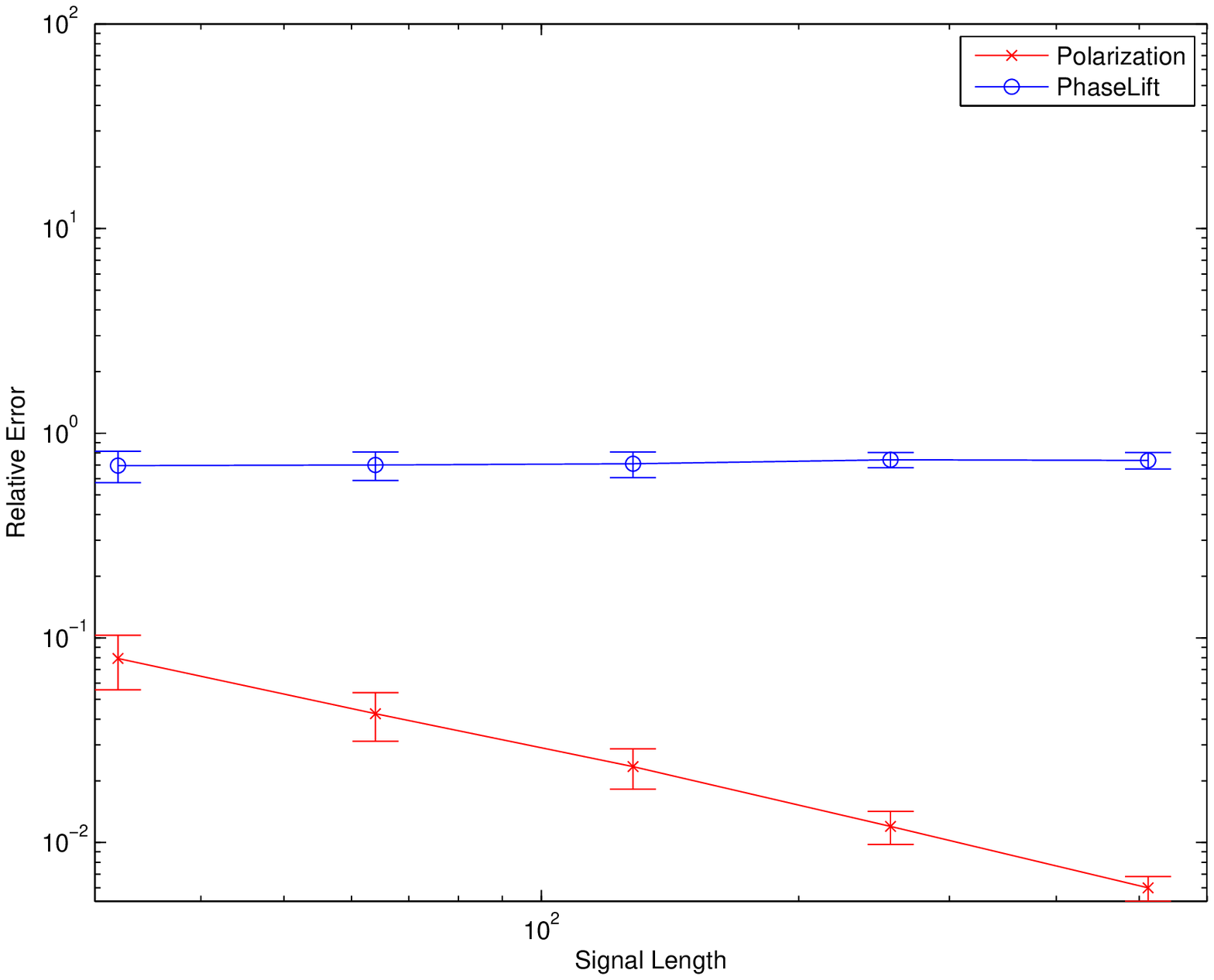}
\includegraphics[width=0.45\textwidth]{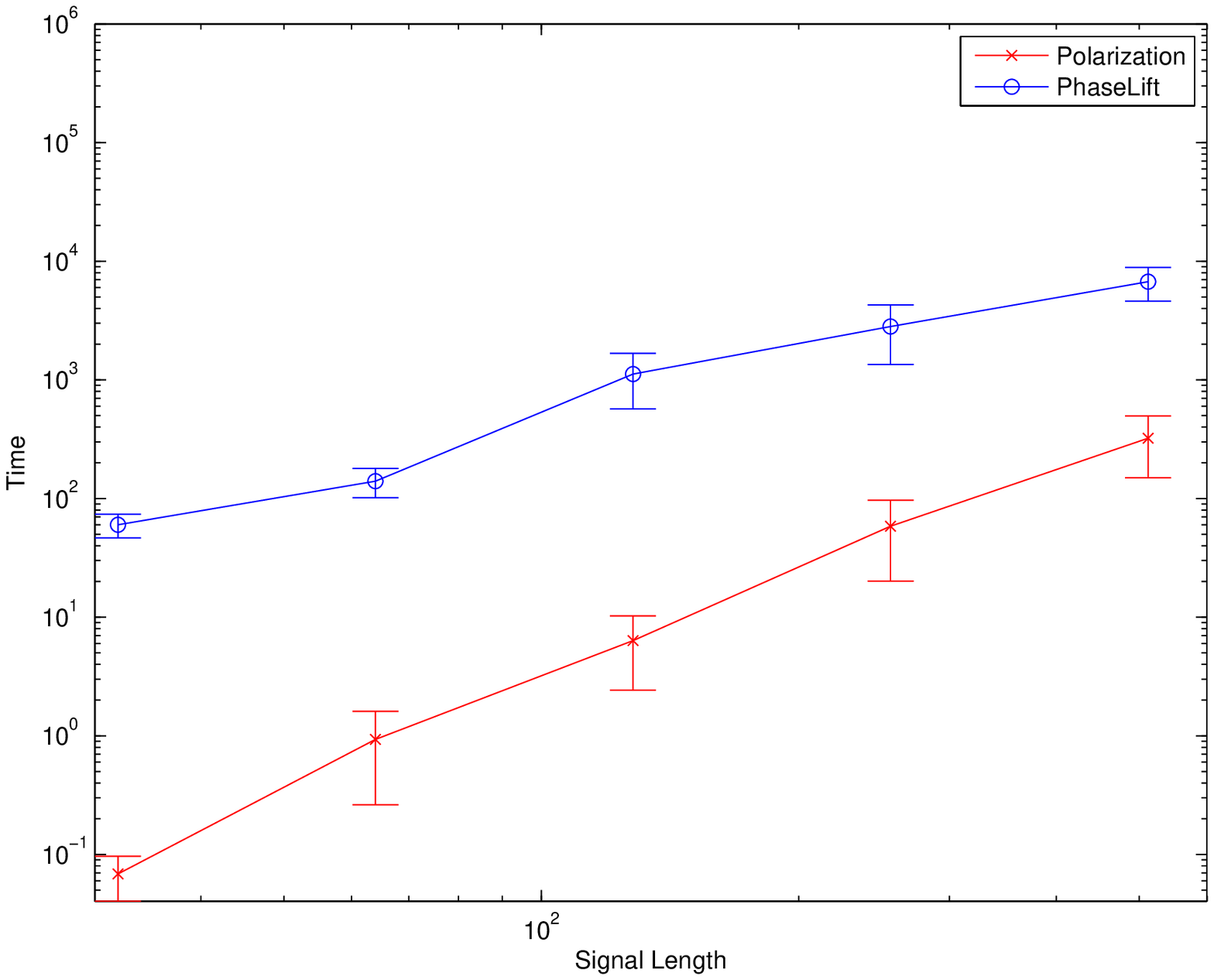}
\includegraphics[width=0.45\textwidth]{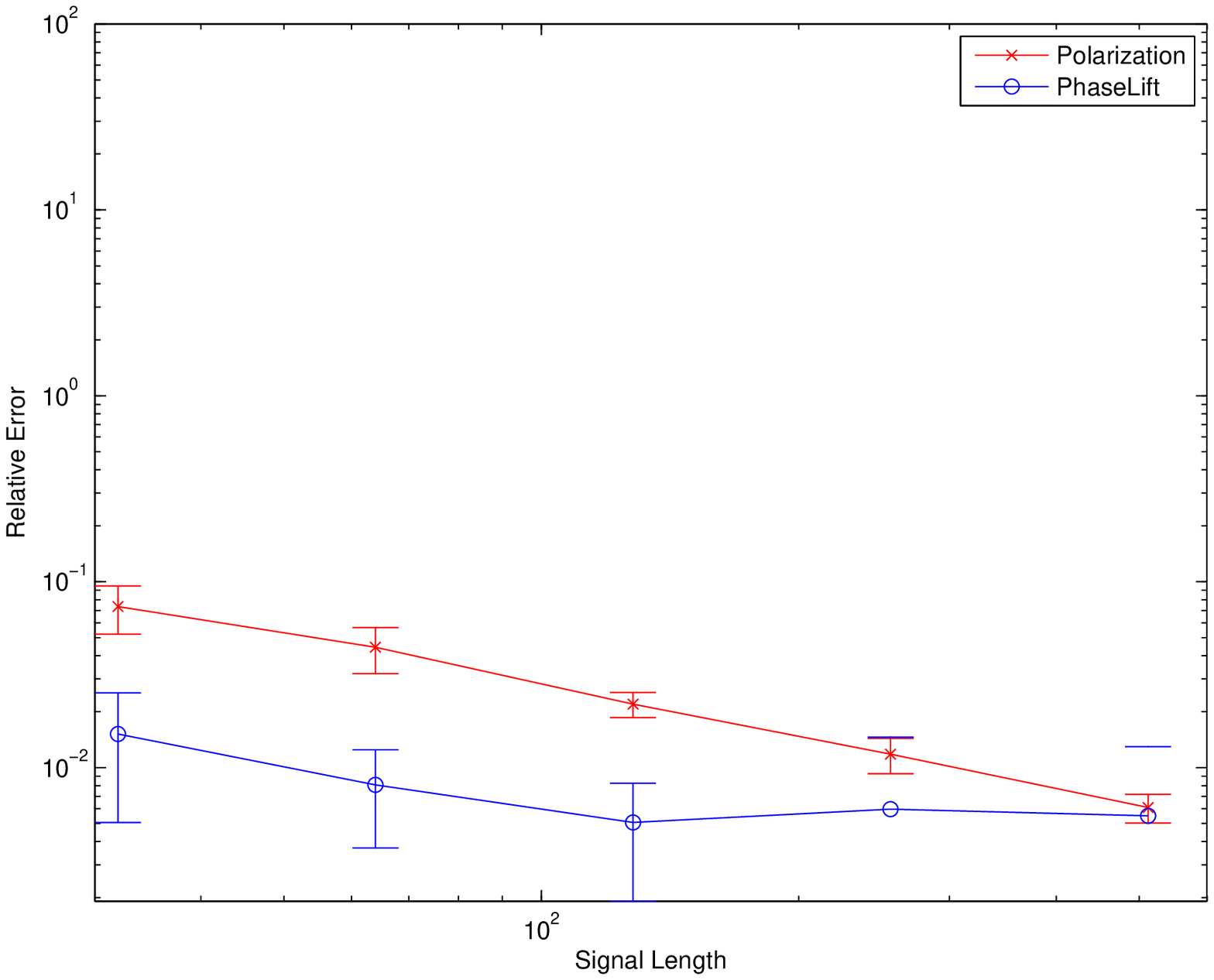}
\includegraphics[width=0.45\textwidth]{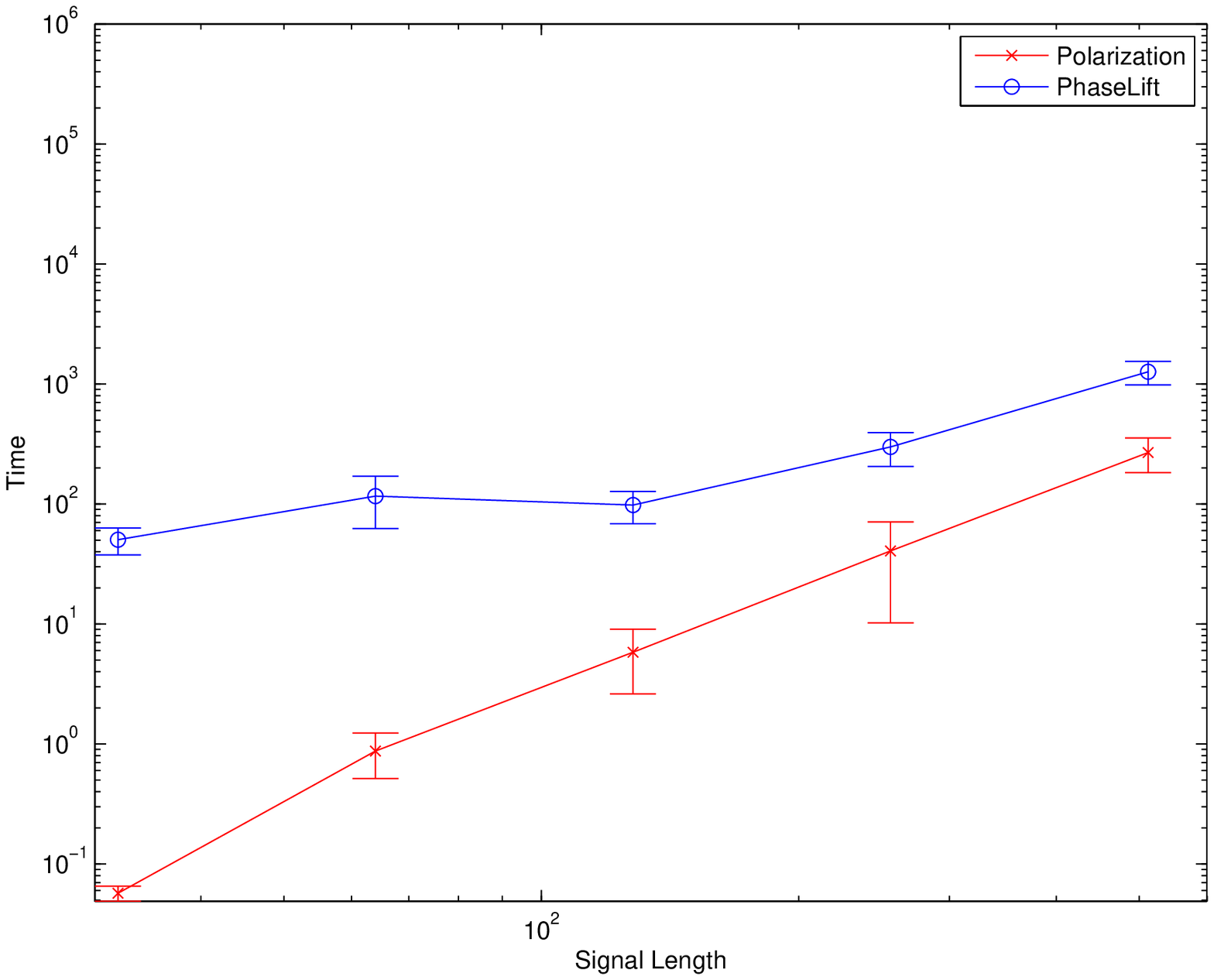}
\includegraphics[width=0.45\textwidth]{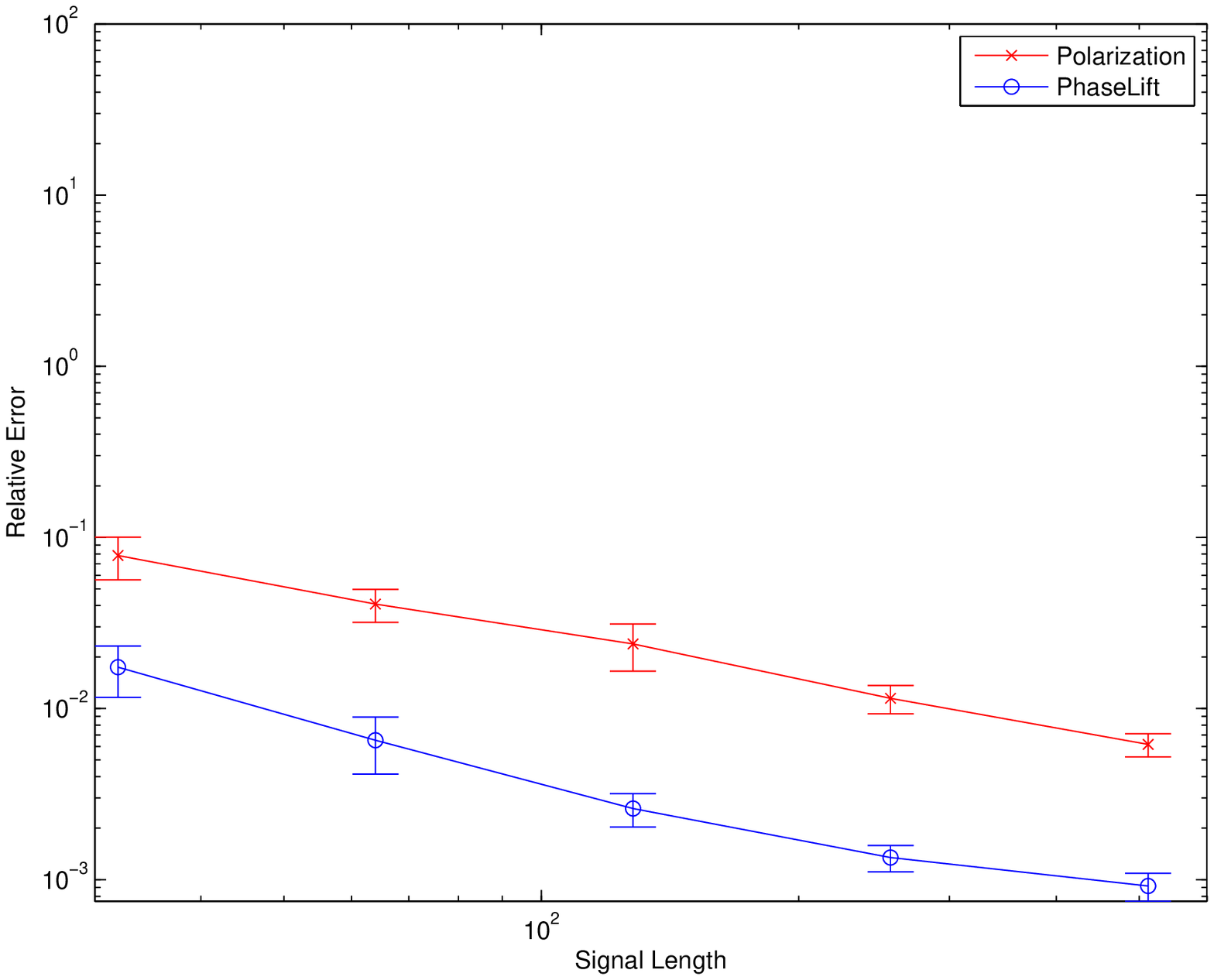}
\caption{
\label{figure.noise10}
Phase retrieval with $\mathcal{N}(0,1)$ noise added to each intensity measurement.
For the first row of graphs, PhaseLift is given the original $K=3$ random masks, for the second, it is given all $18|A|+3$ of the masks given to polarization, and in the last row, PhaseLift is given $18|A|+3$ random masks.
The plotted data illustrate the sample mean plus/minus one sample standard deviation.
In some cases, the lower error bar is negative, and so it is not plotted in the log scale.
}
\end{figure}

Figure~\ref{figure.noise00} presents the results of the noiseless scenarios ($\sigma^2=0$). Here, perhaps the most striking distinction between polarization and PhaseLift is the relative error; indeed, polarization produces an estimate with almost no error ($\sim 10^{-14}$), whereas PhaseLift produces noticeable error ($\sim 10^{-2}$). In addition, polarization beats PhaseLift when using the same number of masks. For higher dimension signals, this advantage translates to minutes versus hours of running time. Letting $N$ denote the number of intensity measurements, then PhaseLift takes $\mathcal{O}(N^{3.5})$ operations, whereas the computational bottleneck in polarization is the least-squares step, which takes $\mathcal{O}(N^3)$ operations, but this does not account for the disparity in run time.
Rather, the disparity is indicative of a large constant in front of the $N^{3.5}$.
In comparing the relative error in reconstruction, we note that PhaseLift terminates when the step size is sufficiently small, so this accounts for the significant difference in performance.
Overall, in the noiseless case, polarization handily defeats PhaseLift, even when PhaseLift is given the random masks it prefers.

Next, Figure~\ref{figure.noise01} illustrates how polarization and PhaseLift perform in the present of some noise ($\sigma^2=0.1$).
In this case, PhaseLift continues to be rather slow, and it is clear that polarization enjoys greater stability from the $18|A|$ additional masks when compared to PhaseLift's performance with just the original $3$ random masks.
In~\cite{CandesESV:11}, it is claimed that $3$ random masks are sufficient for PhaseLift to reconstruct typical signals (and this is corroborated with our experiments in the noiseless case), but Figure~\ref{figure.noise01} illustrates that PhaseLift requires more masks in order to reconstruct stably.
And indeed, when PhaseLift is given more masks, it performs much better.
Specifically, when they use the same masks, polarization and PhaseLift produce estimates with statistically comparable relative error, and PhaseLift performs slightly better when given the random masks it prefers.
Intuitively, it makes sense that PhaseLift should be at least slightly more stable than polarization since it leverages all of the measurements in a democratic fashion, whereas polarization disenfranchises the vertex and edge measurements that are too small (i.e., ``wasting'' measurements).

Figure~\ref{figure.noise10} presents the results of our high-noise scenario ($\sigma^2=1$), which are very similar to the results corresponding to $\sigma^2=0.1$.

\section{Discussion}

In this paper, we showed how to leverage the polarization-based phase retrieval technique to construct $\Theta(\log M)$ Fourier masks that uniquely determine every $M$-dimensional signal, and then we used numerical simulations to illustrate the stability of polarization with such masks.
At this point, we offer two conclusions:
\begin{itemize}
\item[(i)] The polarization technique for phase retrieval is flexible enough to provably accommodate certain measurement design criteria (such as Fourier masks).
\item[(ii)] If you have the ability to use polarization instead of PhaseLift, you will gain considerable speedups in run time at the price of a slight increase in relative error.
\end{itemize}
It remains to be rigorously proved that stability holds in the Fourier masks setting.
We did not exploit the inherent Fourier structure to further speed up reconstruction, though this could very well be possible.
The number of masks might be decreased if additional information about the signal were leveraged, as in~\cite{Fannjiang:11,JaganathanOH:12}.
Finally, we still seek Fourier masks--based performance guarantees for PhaseLift and its modifications~\cite{DemanetH:12,WaldspurgerAM:12}.

\section*{Acknowledgments}

The authors thank the Erwin Schr\"{o}dinger International Institute for Mathematical Physics for hosting a workshop on phase retrieval that helped solidify the main ideas in this paper.
A.\ S.\ Bandeira was supported by NSF DMS-0914892.
The views expressed in this article are those of the authors and do not reflect the official policy or position of the United States Air Force, Department of Defense, or the U.S.~Government.

\end{document}